\documentclass[12pt,reqno]{amsart}

\usepackage{amsfonts}
\usepackage{amsmath}
\usepackage{amsthm}
\usepackage{amsxtra}
\usepackage{amssymb,latexsym}
\usepackage[mathcal]{eucal}
\usepackage{amscd} 
\usepackage{graphicx}

\usepackage{hyphenat}

\usepackage{tikz,tikz-cd}
\usetikzlibrary{matrix,arrows}
\tikzset{node distance=2cm, auto}

\usepackage{titlesec}
\titleformat{\section}{\bfseries}{\thesection.}{1em}{\filcenter}{}
\titlespacing{\section}{0pt}{24pt}{6pt}

\newtheorem{theo}{Theorem}[section]
\newtheorem{prop}[theo]{Proposition}
\newtheorem{lem}[theo]{Lemma}
\newtheorem{cor}[theo]{Corollary}

\theoremstyle{definition}
\newtheorem{df}[theo]{Definition}

\usepackage{etoolbox}
\newtheorem{ex}[theo]{Example}
\AtEndEnvironment{ex}{\null\hfill\qedsymbol}

\newtheoremstyle{dotless}{}{}{}{}{\bfseries}{}{ }{}
\theoremstyle{dotless}
\newtheorem*{remark}{Remark:}

\newtheorem*{qs}{Questions:}

\def \ol {\overline}
\def \N {\mathbb N}

\def \Z {\mathbb Z}
\def \R {\mathbb R}

\def \NN {\mathcal N}

\def \RR {\mathcal R}

\def \00 {\mathbf 0}
\newcommand \A {\mathcal A}
\newcommand \BB {\mathcal B}
\newcommand \CC {\mathcal C}
\newcommand \II {\mathcal I}
\newcommand \OO {\mathcal O}
\newcommand \YY {\mathcal Y}
\newcommand{\al}{\alpha}

\newcommand{\ga}{\gamma}
\newcommand{\del}{\delta}

\newcommand{\ep}{\epsilon}

\newcommand{\om}{\omega}
\newcommand{\Om}{\Omega}

\newcommand{\Iso}{\mathrm{Iso}}
\newcommand{\Trans}{\mathrm{Trans}}
\newcommand{\wh}[1]{\widehat{#1}}
\newcommand{\wt}[1]{\widetilde{#1}}

\numberwithin{equation}{section}

\begin{document}

\title{ Chain Transitive Homeomorphisms on a Space: All or None}

\author{Ethan Akin and Juho Rautio}
\address{Mathematics Department,
 The City College, 137 Street and Convent Avenue,
 New York City, NY 10031, USA}

\email{ethanakin@earthlink.net}

\address{University of Oulu, Department of Mathematical Sciences, PB 8000, FI-90014 Oulun yliopisto, Finland.}
\email{Juho.Rautio@oulu.fi}

\date{June, 2015}

\begin{abstract}
We consider which spaces can  be realized as the omega limit set of the discrete time dynamical system.
This is equivalent to asking which spaces admit a chain transitive homeomorphism and which do not. This
leads us to ask for spaces where all homeomorphisms are chain transitive.
\end{abstract}



\maketitle

\section{Introduction}\label{intro}

We will use the term \emph{space} to mean a nonempty, compact metrizable space unless otherwise mentioned.
When a metric is required we, assume that one is chosen and fixed. Where a metric is
used in a construction below, the result is independent of the choice of metric. Broadly, our metric space ideas and
constructions
are really uniform space concepts and a compact space has a unique uniformity.

The dynamical systems we will consider are pairs $(X,f)$, where $f: X \to X$ is a continuous map on a  space $X$.
If $x \in X$, then the sequence $\{ f(x), f^2(x), \dots \}$ is the \emph{positive $f$-orbit sequence} of $x$ and
$\om f(x)$ is the set of limit points of the sequence.

In a series of papers, \cite{ABCP}, \cite{AC1} and \cite{AC2}, Agronsky and his collaborators initiated study of the
\emph{omega set problem}. Call the pair
$(X,f)$ an \emph{omega system} when it can be embedded as a subsystem of some $(Y,g)$ so that $X = \om g(y)$ for some
$y \in Y$. The authors call the system \emph{orbit enclosing} when one can choose $(Y,g) = (X,f)$, i.e. there exists $x \in X$
such that $\om f(x) = X$. A space $X$ is an \emph{omega set} when there exists a continuous map $f$ on $X$ so that
$(X,f)$ is an omega system.

In current parlance, the system $(X,f)$ is an orbit enclosing omega system exactly when it is topologically transitive.
In general, the system $(X,f)$ is an omega system exactly when it is chain transitive.

Given $\ep \geq 0$, a finite or infinite sequence $\{ x_n \in X \}$ with at least two terms is an \emph{$\ep$-chain} for $(X,f)$ if
$d(f(x_k),x_{k+1}) \leq \ep$ for all terms $x_k$ of the sequence (except the last one). The system
$(X,f)$ is called \emph{chain transitive} when every pair of points of $X$ can be connected by some finite $\ep$-chain
for every positive $\ep$. A subset $A \subset X$ is called a \emph{chain transitive subset} when it is closed and
$f$-invariant 
and the subsystem $(A,f)$ is chain transitive.

It is well known that any omega limit set is a chain transitive subset,  e.g. \cite{A-93} Proposition 4.14.
On the other hand,
as observed by Takens, it is an easy exercise to show that if $(X,f)$ is chain transitive, then then it can be embedded
in a larger system in which it is an omega limit set, \cite{A-93} Exercise 4.29. We will review the proofs in Section 3.
The construction uses a subset $Y $ of $X \times [0,1]$. In particular, if $X \subset [0,1]^n$ then $Y \subset [0,1]^{n+1}$.
Applying the Tietze Extension Theorem in each coordinate, we can extend $g$ to a continuous map on all of $[0,1]^{n+1}$ and so
obtain $X$ as the omega limit set of a system on $[0,1]^{n+1}$. Note, however, that even if $g$ is a homeomorphism, we might
not be able to extend it to a homeomorphism on $[0,1]^{n+1}$.

This equivalence is really a classical result due to Dowker and Friedlander \cite{DF} for homeomorphisms and to
Sarkovskii \cite{Sa} in general. They, however, use a different but equivalent condition instead of chain transitivity.
They say that $(X,f)$ is \emph{$f$-connected} if, for any proper, nonempty, closed subset $U \subset X$, the intersection
$f(U) \cap \ol{X \setminus U} \not= \emptyset$. They show that $(X,f)$ is an omega system iff it is
$f$-connected.

To clarify the relationship between chain transitivity and $f$\hyp{}connectedness, we recall the concept of an \emph{attractor},
as described by Conley in \cite{C} and with detailed exposition in \cite{A-93}. Call a closed set $U \subset X$ an \emph{inward set}
for $f$ if $f(U)$ is contained in the interior $U^{\circ}$ or, equivalently, $f(U) \cap \ol{X \setminus U} = \emptyset$.
Thus,  $(X,f)$ admits a proper, nonempty, inward subset iff it is not $f$-connected.   If $U$ is an inward set, then
$A = \bigcap_{n=0}^{\infty} \ f^n(U)$ is called the associated attractor. For a number of equivalent descriptions of an attractor
see \cite{A-93} Theorem 3.3. Theorem 4.12 of \cite{A-93} says that $(X,f)$ is chain transitive iff $X$ is the
only nonempty attractor and iff $X$ is the only nonempty inward set.  It follows that chain transitivity and
$f$-connectedness are equivalent concepts.

The label ``attractor'' has been used for other ideas. Some authors refer to all omega limit sets as attractors.
An attractor as defined above need not be chain transitive, and so there are attractors which are not omega systems.
A more reasonable definition is that a set $A$ is an attractor for $f$ when it is $f$-invariant and, for every $x$ in some
neighborhood of $A$, we have $\om f(x) \subset A$.  This condition is necessary in order that $A$ be an attractor \`{a} la Conley,
but it is not sufficient. If $X$ is the one-point compactification of $\Z$ and $f$ is the extension to $X$ of translation by
$1$ on $\Z$, then $(X,f)$ is chain transitive.  On the other hand, the point at infinity is the omega limit set of every point.
By Theorem 3.6(a) of \cite{A-93}, an $f$-invariant subset $A$ is an attractor iff $\{ x : \om f(x) \subset A \}$ is
a neighborhood of $A$ and,
in addition, $A$ is stable, i.e. for every $\ep > 0$ there exists $\del > 0$ such that if $x$ is $\del$-close to $A$,
then the forward orbit of $x$  remains $\ep$-close to $A$.

 The identity map $1_X$ on $X$ is chain transitive iff $X$ is connected.  Thus, a connected space admits
 a chain transitive homeomorphism and so is an omega space.

Recall that a \emph{Peano space} is a compact, connected, locally connected space or, equivalently, a continuous image of the
unit interval. The major result of \cite{AC2} is that if $X$ has finitely many components
and each is a nontrivial, finite-dimensional
Peano space, then $X$ is an orbit enclosing omega space.  That is, it admits a topologically transitive map.

Our purpose here is to consider the homeomorphism version of the omega set problem.  That is, we ask when a space $X$
admits a homeomorphism $f$  such that $(X,f)$ is an omega system or, equivalently, such that $f$ is a chain transitive map.

To illustrate the difference between the original problem and the homeomorphism version, consider the \emph{tent map}
on $[0,1]$ defined by
\begin{equation}\label{001a}
T(t) \ = \
\begin{cases}
2t & \text{for } 0 \leq t \leq \frac{1}{2}, \\
1 - 2t &  \text{for } \frac{1}{2} \leq t \leq 1,
\end{cases}
\end{equation}
which is well known to be topologically transitive. Now let $X_0 = [0,1] \times \{ 0, 1\}$, and define $f_0$ on $X_0$
by
\begin{equation}\label{001b}
f_0(t,0) \ = \ (t,1) \qquad \text{ and } \qquad f_0(t,1) \ = \ (T(t),0),
\end{equation}
so that $f_0^2 = T \times 1_{\{ 0, 1 \}}$.

Let $(X,f)$ be the quotient system obtained by identifying the points $(0,1) = (1,1)$ in $X$. Thus, $X$ is the disjoint
union of a circle and an interval. It easily follows that $X$ does not admit a chain transitive homeomorphism.  On the other hand,
$f$ is a topologically transitive map.

Let $(X_1,f_1)$ be the quotient system obtained by identifying the points $(0,0) = (0,1) = (1,1)$ in $X_1$.
Thus, $X_1$ consists of a circle and an interval joined at a point.  As $X_1$ is connected, the identity is a chain
transitive homeomorphism. Since the points of the interval other than $(1,0)$ separate the space and the points of the circle
other than the intersection point with the interval do not, it easily follows that $X_1$ does not admit a topologically
transitive homeomorphism.  On the other hand, $f_1$ is a topologically transitive map.

%

 If $X$ contains a proper, clopen, nonempty, $f$-invariant set $A$, then we say that  $X$ is \emph{$f$-decomposable}.
 In this case $(X,f)$ is not chain transitive, for if $\ep$ is smaller
 than the distance from $A$ to its complement, then any $\ep$-chain which begins in $A$ remains in $A$.
 With $H(X)$ the group of homeomorphisms on $X$, we say that $X$ is \emph{$H(X)$-decomposable} if there is a proper, clopen, nonempty subset
 $A$ of $X$ such that $A$ is invariant for every homeomorphism on $X$. Clearly, if $X$ is $H(X)$-decomposable, then it
 admits no chain transitive homeomorphism.  For example, let $\Iso(X)$ denote the (possibly empty) set of isolated points
 of $X$. If the closure $\ol{\Iso(X)}$ is a proper, clopen, nonempty subset of $X$, then $X$ is $H(X)$-decomposable and
 so admits no chain transitive homeomorphism. In the zero-dimensional case, this is the only obstruction.
 We prove a slightly more general result in Section \ref{trans}.


 If $X$ is not $f$-decomposable (or not $H(X)$-decomposable), we will call it \emph{$f$-indecomposable} (resp. \emph{$H(X)$-indecomposable}).

 \begin{theo}\label{theo1.01}   If $X$ is a space such that $\ol{\Iso(X)}$ is not a proper clopen subset of $X$ and
 such that the open set $X \setminus \ol{\Iso(X)}$ is empty or zero-dimensional, then $X$ admits a chain transitive
 homeomorphism. \end{theo}

\begin{cor}\label{cor1.02} If the isolated points are dense in $X$, then $X$ admits a chain transitive homeomorphism.\end{cor}

Thus, the problems which remain come from the nontrivial components. A clopen component is called an \emph{isolated component}.
Clearly, if the closure of the isolated components is a proper, clopen, nonempty subset of $X$, then $X$ is $H(X)$-decomposable.

If $\YY$ is a set of connected spaces, let $C_{\YY}$ denote the closure
of the union of those components of $X$ which are homeomorphic to some element of
$\YY$. If $C_{\YY}$ is a proper, clopen, nonempty subset, then $X$ is
$H(X)$-decomposable.  For example, if $X$ has finitely many components, then either they are all homeomorphic, in which case $X$
admits a periodic chain transitive map, or $X$ is $H(X)$-decomposable. Contrast this with the map result described above.

We say that $X$ satisfies the \emph{diameter condition on isolated components} if for
every $\ep > 0$ there are only finitely many isolated
components with diameter greater than $\ep$.

\begin{theo}\label{theo1.03} If $X$ satisfies the diameter condition and the union of the isolated
components is dense in $X$, then either $X$ is $H(X)$-decomposable or else $X$ admits a chain transitive homeomorphism.
\end{theo}


The rest of Section \ref{trans} consists of counterexamples to reasonable conjectures. We construct
\begin{itemize}
\item A space $X$ that  is $H(X)$-decomposable, with all components homeomorphic, and no isolated components.

\item A space $X$ that  is $H(X)$-indecomposable but $f$-decomposable for every $f \in H(X)$.
The space can be chosen with the isolated components all homeomorphic and with a dense union.

\item A space $X$  that  is $f$-indecomposable for some $f \in H(X)$ but admits no chain transitive homeomorphism.
The space can be chosen with the isolated components all homeomorphic and with a dense union.
\end{itemize}

These examples rule out the obvious extension of the above corollary. The isolated components can be
dense and all homeomorphic to one another, but nonetheless the space admits no chain transitive homeomorphism.

Having considered when there are no chain transitive homeomorphisms, we consider in Section \ref{allct}
the question:
When is every homeomorphism on a space $X$ chain transitive? For such a space $X$, the identity map $1_X$ is chain transitive,
and so $X$ must be connected.

In \cite{dG-W}, rigid spaces were defined and Peano space examples  were
constructed.  A space $X$ is \emph{rigid} if $1_X$ is the only homeomorphism on $X$, i.e. the homeomorphism group
$H(X)$ is trivial. For a connected rigid space, it is trivially true that all homeomorphisms are chain transitive.

Using rigid spaces one can construct more interesting examples. Following \cite{dG}, we can begin with a finitely generated
group $G$ and use rigid spaces instead of intervals as edges in the Cayley graph. If $X$ is the one-point compactification
of this fattened Cayley graph, then $H(X)$ is isomorphic to $G$ and every homeomorphism is chain transitive.
We obtain examples with non-discrete homeomorphism group and even with the path components nontrivial.

In these cases, the homeomorphism group does not act in a topologically transitive manner on the space.  Distinct points in
each rigid piece are homeomorphically distinct. It is possible to obtain examples with all homeomorphisms chain transitive and
with the homeomorphism group acting in a topologically transitive manner.  These are built using the recent, beautiful construction in
\cite{D-S-T} of Slovak spaces.  A \emph{Slovak space} $X$ has $H(X)$ isomorphic to $\Z$ and every element in it other than
the identity $1_X$ acts minimally on $X$.  We call a space \emph{Slovakian} if every homeomorphism other than $1_X$ is topologically
transitive. Using such Slovakian spaces we construct a space $X$ such that $H(X)$ is topologically transitive on $X$, every
element of $H(X)$ is chain transitive, and the homeomorphism group of the Cantor set occurs as a closed, topological subgroup
of $H(X)$.

\section{Relation Dynamics}\label{rel}

It will be convenient to  use the dynamics of closed relations, and so we briefly review the ideas from
\cite{A-93}. Recall that our spaces $X, Y,$ etc. are assumed to be nonempty, compact, metrizable spaces with a fixed metric
chosen when necessary.

For spaces $X, Y$, a  relation $R : X \to Y$ is a  subset of
 $X \times Y$. The set $R$ is a \emph{relation on $X$} when $Y = X$.
   A map is a
 relation such that $R(x) = \{ y : (x,y) \in R \}$ is a singleton set for every
 $x \in X$.  For $A \subset X$, the \emph{image} $R(A)$ is defined to be $\bigcup_{x \in A} \ R(x)$.
 Equivalently, $R(A)$ is the projection to $Y$ of $R \cap (A \times Y) \subset X \times Y$.
The inverse $R^{-1} : Y \to X$ is defined to be $\{ (y,x) : (x,y) \in R \}$. For $B \subset Y$,
 we let $R^*(B) = \{ x \in X : R(x) \subset B \} = X \setminus R^{-1}(Y \setminus B)$.
 So $R^*(B) \subset R^{-1}(B) \cup R^*(\emptyset)$.
 If $R$ is a map, then $R^*(B) = R^{-1}(B)$.

 For example,
 $\bar V_{\ep} = \{ (x,y) \in X \times X : d(x,y) \leq \ep \}$ is a relation on $X$
 with $\bar V_{\ep}(x)$ the closed ball of radius $\ep$ and center $x$.  When $\ep = 0$,
 $\bar V_{\ep} \ = \ 1_X$, the identity map on $X$.

  If $R : X \to Y$ and $S : Y \to C$ are relations, then the composition $S \circ R : X \to C$ is the image under the
 projection to $X \times C$ of the set $(R \times C) \cap (X \times S) \subset X \times Y \times C$.
 Thus, $(x,c) \in S \circ R$ iff there exists $y \in Y$ such that $(x,y) \in R$ and $(y,c) \in S$.
 Composition is associative, and $(S \circ R)^{-1} \ = \ R^{-1} \circ S^{-1}$.

 A relation $R$ on $X$ is \emph{reflexive} when $1_X \subset R$, \emph{symmetric} when $R^{-1} = R$, and
 \emph{transitive} when $R \circ R \subset R$.

 For a relation $R$ on $X$, we let $R^{n+1} = R^n \circ R$ and $R^{-n} = (R^{-1})^n$ for $n = 1,2,...$,
 and let $R^0$ be the identity $ 1_X$. We define the \emph{cyclic set} $|R| = \{ x : (x,x) \in R \}$.

 For a relation $R$ on $X$ a subset $A$ of $ X$ is called
 \emph{forward} $R$-\emph{invariant} (or $R$-\emph{invariant})  if $R(A) \subset A$ (resp. $R(A) = A$).

 For a relation $R$ on $X$, the \emph{orbit relation} is $\OO R = \bigcup_{n = 1}^{\infty} \ R^n$, and the
 \emph{orbit closure relation} $\RR R$ is defined by $\RR R(x) = \ol{\OO R(x)}$ for all $x \in X$.
 The \emph{wandering
 relation} is $\NN R = \ol{\OO R}$. Even when $R$ is a continuous map, $\RR R$ is usually not closed and
 so is a proper subset of $\NN R$.

The  \emph{chain relation} is
\begin{equation}\label{2.01aa}
\CC R \ = \  \bigcap_{\ep > 0} \ \OO (\bar V_{\ep} \circ R \circ \bar V_{\ep}).
\end{equation}

  Both  $\OO R$ and $\CC R$ are transitive relations. Since $(R^n)^{-1} = (R^{-1})^n$, it follows that
  $\OO (R^{-1}) = (\OO R)^{-1}$, $\NN (R^{-1}) = (\NN R)^{-1}$ and $\CC (R^{-1}) = (\CC R)^{-1}$.
  These operators on relations are monotone,
  i.e. they preserve inclusions, and $\OO$ and $\CC$ are idempotent.  That is,
  \begin{equation}\label{2.01}
  \OO (\OO R) \ = \ \OO R \qquad \text{ and } \qquad \CC (\CC R) \ = \ \CC R.
  \end{equation}

It then follows that
\begin{equation}\label{2.01a}
\begin{gathered}
\RR (\OO R)(x) \ = \ \ol{\OO (\OO R)(x)} \ = \ \ol{\OO R(x)} \ = \ \RR R(x), \\
\NN (\OO R) \ = \ \ol{\OO (\OO R)} \ = \ \ol{\OO R} \ = \ \NN R, \\
\CC R \ \subset \ \CC (\OO R) \ \subset \  \CC (\CC R) \ = \ \CC R.
\end{gathered}
\end{equation}

 If $R$ is a transitive relation on $X$, then $R \cap R^{-1}$ is a symmetric, transitive relation which restricts
 to an equivalence relation on $|R|$. We call the equivalence classes the \emph{basic sets} of $R$.

A \emph{closed relation} $R$  is a closed subset of $X \times Y$.
 A map is continuous iff it is a closed relation.  The composition of closed relations is
 closed, and the image of a closed set by a
 closed relation is closed. So if $B$ is open in $Y$, then $R^*(B)$ is open in $X$. For any relation $R$,
 the extensions $\NN R$ and $\CC R$ are closed relations.

 It is easy to see that $\CC R = \CC \ol{R}$, where $\ol{R}$ is the closure of $R$.
 If $R$ is a closed relation, then
 $ \CC R = \bigcap_{\ep > 0} \ \OO (\bar V_{\ep} \circ R)$, see \cite{A-93} Proposition 1.18.
 If $R$ is a closed relation on
 $X$, then $|R|$ is a closed subset of $X$. If $R$ is a closed, transitive relation,
 the basic sets $\{ R(x) \cap R^{-1}(x) : x \in |R| \}$ are closed.

 For a sequence $\{ A_n \}$ of closed sets, $\limsup \ \{ A_n \} \ = \ \bigcap_n \ol{\bigcup_{i \geq n} \ A_i}$.
 We have the identity
 $\ol{\bigcup_n \ A_n} \ = \  (\bigcup_n \ A_n) \ \cup \ \limsup \ \{ A_n \}$. If $R$ is a closed relation on $X$ and $A$ is
 a closed subset of $X$,
 then we define
 $\om R[A] \ = \ \limsup \ \{ R^n(A) \}$.

 If $A$ is forward $R$-invariant and $R$ is closed, then the sequence of closed sets $\{ R^n(A) \}$
 is decreasing and $\om R[A]$ is the intersection. Furthermore, if $y \in \om R[A]$, then $R^{-1}(y)$ meets every $\{ R^n(A) \}$,
 and so by compactness it meets $\om R[A]$ itself.  It follows that $\om R[A]$ is the maximum $R$-invariant
 subset of the closed, forward $R$-invariant set $A$.

 If $R$ is closed, then for $x \in X$ we let
 $$\om R(x) \ = \ \om R[x] \ = \ \limsup \ \{ R^n(x) \},$$
  defining the \emph{omega limit point relation}
 $\om R$ on $X$. We also define $\Om R = \ \limsup \ \{ R^n \}$ for a general relation $R$, so
  $\Om R$ is a closed relation.
When $R$ is closed, we have the following identities:
\begin{equation}
\label{2.01b}
\RR R \ = \  \OO R \cup \om R \qquad \text{ and } \qquad \NN R \ = \ \OO R \cup \Om R.
\end{equation}

 Since $\CC R$ is closed and transitive, the sequence $\{ (\CC R)^n \}$ is decreasing with intersection $\Om \CC R \ = \ \om \CC R$.

 If $R$ is closed, then the following useful identities hold for the chain relation:
 \begin{equation}\label{2.02}
 \begin{gathered}
 \CC R \ = \ \OO R \cup \Om \CC R,\\
 R \cup (\CC R) \circ R \ = \ \CC R \ = \ R \cup R \circ (\CC R),
 \end{gathered}
 \end{equation}
 see \cite{A-93} Proposition 2.4 (c) and Proposition 1.11 (d).

 It is not usually true that $(\om R)^{-1} = \om (R^{-1})$.  We write $\al R$ for $\om (R^{-1})$,
 defining the \emph{alpha limit point relation}.

 The points of $|\OO R|$ are called \emph{periodic points} for $R$, $|\RR R|$ are the \emph{recurrent points}, $ |\NN R| $ are the
 \emph{non-wandering points}, and $  |\CC R|$ are the \emph{chain recurrent points}.
 When $R$ is a closed relation, we can apply the inclusion $|\OO R| \subset |\om R|$ to \eqref{2.01b}
 and \eqref{2.02} to obtain $|\RR R| = |\om R|$, $|\NN R| = |\Omega R|$ and $|\CC R| = |\Omega \CC R|$.
 We call $x$ a \emph{transitive point} for $R$ when $\RR R(x) = X$. We denote by $\Trans(R)$
 the (possibly empty) set of transitive points.

 On $|\CC R|$, $\CC R \cap \CC R^{-1}$ is a closed equivalence relation.  We call the equivalence classes, i.e.
 the basic sets for $\CC R$, the \emph{chain components} of $R$.

 For a  relation $R$ on $X$, we will say that $R$ is
 \emph{minimal} when $\RR R = X \times X$,
 \emph{topologically transitive}  when $\NN R = X \times X$, and
 \emph{chain transitive} when
 $\CC R = X \times X$. We call $R$ \emph{central}  when $|\NN R| = X$ and \emph{chain recurrent} when $|\CC R| = X$.
 From (\ref{2.01a}) it follows that any one of these properties holds for $R$ iff it holds for $\OO R$.

Observe that if $R$ is minimal, then $X$ contains no proper closed, forward $R$-invariant subset.
If $R$ is a continuous map, then the converse is true as well since $\om R(x)$ is
then $R$-invariant and nonempty for every $x$. Thus, if $R$ is a continuous map, it is minimal iff $X = \Trans(R)$.

 A set $U$ is
\emph{inward} for a closed relation $R$ on $X$ if $U$ is closed and $R(U) \subset U^{\circ}$, the interior of $U$.
Since $R(U)$ has a positive distance from the complement of $U$, it easily follows that $U$ is inward for
$\CC R$ when it is inward for $R$. An inward set is therefore forward $\CC R$-invariant. In general, a closed set
$A$ is forward $\CC R$-invariant iff it has a neighborhood base consisting
of inward sets, see \cite{A-93} Theorem 3.3(c).
If $U$ is an inward set, then $A_+ = \om R[U] = \bigcap_{n \in \N} \ R^n(U)$ is the
associated \emph{attractor}. The set $X \setminus U^{\circ}$ is
inward for $R^{-1}$, and  $A_- = \om (R^{-1})[ X \setminus U^{\circ}]$ is the
\emph{dual repellor} for $R$. The pair $A_+, A_-$ is called an \emph{attractor-repellor pair}.
If $x \in X \setminus (A_+ \cup A_-)$, then $\Om \CC R(x) \subset A_+$ and
$\Om \CC R^{-1}(x) \subset A_-$,
see \cite{A-93} Proposition 3.9.

Notice that a clopen set is inward iff it is forward invariant. If an inward set is invariant, then it is clopen.

\begin{prop}\label{prop2.01} Let $R$ be a closed  relation on $X$.
\begin{itemize}
\item[(a)] If $A_+, A_-$ is an attractor-repellor pair, then $|\CC R| \subset A_+ \cup A_-$  and
the intersections $A_+ \cap |\CC R|$ and $A_- \cap |\CC R|$ are each clopen in $|\CC R|$.

\item[(b)] If $x \in X \setminus |\CC R|$, then there exists an attractor-repellor pair $A_+, A_-$ such that
$x \not\in A_+ \cup A_-$.

\item[(c)] If $x, y \in |\CC R|$, then $y \in \CC R(x)$ iff, for all attractors $A$, $x \in A$ implies $y \in A$.

\item[(d)] The space of chain components, i.e. the quotient space of $|\CC R|$ by the equivalence relation
$\CC R \cap \CC R^{-1}$, is a compact zero\hyp{}dimensional metric space.
\end{itemize}
\end{prop}

\begin{proof}
(a) If $x \in X \setminus (A_+ \cup A_-)$, then $\Omega \CC R(x) \subset A_+$, and so
$x \not\in \Omega \CC R(x)$. If $U$ is an inward set with $\om R[U] = A_+$, then $U^{\circ} \cap |\CC R| = A_+ \cap |\CC R|$.

(b), (c) These are part of Proposition 3.11 of \cite{A-93}.

(d) From (b) and (c) applied to $R$ and $R^{-1}$ it follows that every chain component $C$ is the intersection of
$\{ A \cap |\CC R| : A$ is an attractor or repellor and $C \subset A \}$. Hence, the attractors and repellors induce
a clopen subbase on the space of chain components.
\end{proof}

For a relation $R$ on a space $X$ we say that $X$ is $R$-\emph{decomposable} if there is a
proper, forward $R$-invariant decomposition $\{A_1, A_2 \}$ of $X$, i.e.
 each is a nonempty, closed, forward $R$-invariant subset and $A_1 \cap A_2 = \emptyset, \ A_1 \cup A_2 = X$.
 Such a pair of proper clopen sets is
 called an $R$-\emph{decomposition}. Since the sets are clopen and forward $R$-invariant,
 they are inward for $R$.  Hence, an $R$-decomposition  is a $\CC R$-decomposition.
 So the notion of decomposability is the same for $R, \OO R, \RR R, \NN R$ and $\CC R$.

 If no such a decomposition exists, $X$ is said to be
 $R$-\emph{indecomposable}.

For any relation $R$ on $X$ let $R_{\pm} = R \cup 1_X \cup R^{-1}$. This is a reflexive
and symmetric  relation on $X$ so that $\OO (R_{\pm})$ and
$\CC (R_{\pm})$ are equivalence relations on $X$.

\begin{prop}\label{prop2.02} For a relation $R$ on $X$, the following conditions are equivalent:
\begin{itemize}
\item[(i)] The space $X$ is $R$-indecomposable.
\item[(ii)] The space $X$ is $R_{\pm}$-indecomposable.
\item[(iii)] The relation $R_{\pm}$ is chain transitive.
\end{itemize}
\end{prop}

\begin{proof}
(i) $\Leftrightarrow$ (ii): If a set is forward $R$-invariant,
then its complement is forward $R^{-1}$-invariant.
Hence, an $R$-decomposition  is an $R_{\pm}$-decomposition. Clearly, the reverse is true since $R \subset R_{\pm}$.

(iii) $\Rightarrow$ (ii): If $\{A_1, A_2 \}$ is a decomposition, then each is inward for $R_{\pm}$ and
therefore forward $\CC R_{\pm}$-invariant.
Hence, $R_{\pm}$ is not chain transitive.

(ii) $\Rightarrow$ (iii): If $R_{\pm}$ is not chain transitive, then the equivalence
relation $\CC R_{\pm}$ on $X$ has more than one equivalence
class.  By Proposition \ref{prop2.01} (d) the space of equivalence classes is zero-dimensional.
Hence, there is a pair of disjoint, nonempty,
clopen sets $\{A_1, A_2 \}$ which cover $X$ such that each is a union of equivalence classes. Since each equivalence class is
$\CC R_{\pm}$-invariant, it follows that $\{ A_1, A_2 \}$ is a $R_{\pm}$-decomposition.
\end{proof}

\begin{df}\label{df2.02a} Let $\pi : X_1 \to X_2$ be a continuous map between spaces, and let $R_i$
be a relation on $X_i$ for $i = 1,2$.  We say that \emph{$\pi$ maps $R_1$ to $R_2$} if $\pi \circ R_1 \subset R_2 \circ \pi$
and that \emph{$\pi$ is a semi-conjugacy from $R_1$ to $R_2$} if
 $\pi \circ R_1 = R_2 \circ \pi$. \end{df}

\begin{prop}\label{prop2.02b} Let $\pi : X_1 \to X_2$ be a continuous map between spaces, and let $R_i$
be a relation on $X_i$ for $i = 1,2$.
\begin{itemize}
\item[(a)] The function $\pi$ maps $R_1$ to $R_2$ iff $(\pi \times \pi)(R_1) \subset R_2$. If $\pi$ is a semi-conjugacy from
$R_1$ to $R_2$ and $\pi(X_1) = X_2$, i.e. $\pi$ is surjective, then $(\pi \times \pi)(R_1) = R_2$.

\item[(b)] If $R_i$ is a mapping for $i = 1,2$, then $\pi$ is a semi-conjugacy from $R_1$ to $R_2$ if it maps $R_1$ to $R_2$.

\item[(c)] If $\pi$ maps $R_1$ to $R_2$, then $\pi$ maps $R_1^n$ to $R_2^n$ for all $n \in \Z$ and maps
$\A R_1$ to $\A R_2$ for $\A = \OO, \RR, \NN$ and $\CC$.

\item[(d)] Assume $\pi$ is surjective and maps $R_1$ to $R_2$.  If $R_1$ is minimal, topologically transitive, central, chain transitive
or chain recurrent, then $R_2$ satisfies the corresponding property.

\item[(e)] If $\pi$ is a semi-conjugacy from  $R_1$ to $R_2$ and $\pi$ is an open map, then  $\pi$ is a semi-conjugacy from
 $\A R_1$ to $\A R_2$ for $\A = \OO, \RR, \NN$ and $\CC$.
\end{itemize}
\end{prop}

\begin{proof}
(a) Firstly, $\pi$ maps the relation $R_1$ to $R_2$ iff $(x,y) \in R_1$ implies $(\pi(x),\pi(y)) \in R_2$,
 and so the first equivalence
 is clear. The second result is easy to check.

 (b) An inclusion between functions is an equation.

 (c) Given $\ep > 0$, there exists $\del > 0$ such that $(\pi \times \pi)(\bar V_{\del }) \subset V_{\ep }$ since $\pi$ is uniformly
 continuous. If $\{ (x_n,y_n) \in R_1 \}$ is a finite or infinite sequence with $(y_n,x_{n+1}) \in \bar V_{\del }$, then
 $\{ (\pi(x_n),\pi(y_n)) \in R_2 \}$ with $(\pi(y_n),\pi(x_{n+1})) \in \bar V_{\ep }$. That is, $\del$-chains for $R_1$ are
 mapped to $\ep$-chains for $R_2$.  It follows that $\pi$ maps $\CC R_1$ to $\CC R_2$.  The others are easy to check.

 (d) If $\pi$ is surjective and maps $\A R_1$ to $\A R_2$, then $X_1 \times X_1 = \A R_1$ implies $X_2 \times X_2 = \A R_2$,
 and $1_{X_1} \subset \A R_1$ implies $1_{X_2} \subset \A R_2$.

 (e) The function $1_{X_1} \times \pi$ maps $1_{X_1}$ to $\pi \subset X_1 \times X_2$.  If $\pi$ is open and $\ep > 0$, then
$1_{X_1} \times \pi$ maps $\bar V_{\ep}$ to a neighborhood of $\pi$ and so contains $\bar V_{\del } \circ \pi$ for some $\del > 0$.
That is, $ \bar V_{\del } \circ \pi \subset \pi \circ \bar V_{\ep} $.  If $\{ (u_n,v_n) \in R_2 \}$ is a finite or infinite sequence
with $(v_n,u_{n+1}) \in \bar V_{\del }$ and $\pi(x_n) = u_n$, then because $\pi$ is a semi-conjugacy on $R_1$, there
exists $y_n \in X_1$ such that $(x_n,y_n) \in R_1$ and $\pi(y_n) = v_n$, and there exists $x_{n+1}$ such that
$(y_n,x_{n+1}) \in \bar V_{\ep}$ and $\pi(x_{n+1}) = u_{n+1}$. Thus, every $\del$-chain for $R_2$ can be lifted to an
$\ep$-chain for $R_1$ with a given initial lift. That is, $\pi$ is a semi-conjugacy from $\CC R_1$ to $\CC R_2$.
The cases of $\A = \OO, \RR $ and $\NN$ are similar. In fact, for $\OO$ and $\RR$ it is not necessary that the map be open.
\end{proof}

\begin{remark}
Suppose $R_i = 1_{X_i}$ with $X_2 = [0,1]$ and $X_1 = \{ -1 \} \cup [0,1]$, and let $\pi : X_1 \to X_2$ be an extension of
the identity on $[0,1]$, mapping $-1$ to some point of $[0,1]$. The map $\pi$ is a semi-conjugacy from $R_1$ to $R_2$
and maps $\CC R_1$ onto $\CC R_2$, but it is not a semi-conjugacy from $\CC R_1$ to $\CC R_2$.
\end{remark}

As indicated in the Introduction, our concern is with homeomorphisms.  However, we will apply this machinery to
three different relations.

Let $H(X)$ be the homeomorphism group of $X$.  For $f \in H(X)$, we define $f_{\pm}$
as the closed relation $f \cup 1_X \cup f^{-1}$. Clearly, we have
\begin{equation}\label{2.03}
\begin{split}
\OO f \ &= \ \{ f^n : n \in \N \}, \\
\OO f_{\pm} \ &= \ \{ f^n : n \in \Z \} \ = \ \OO f \cup 1_X \cup \OO f^{-1}, \\
\RR f(x) \ &= \ \ol{\{ f^n(x) : n \in \N \}}, \\
\RR f_{\pm}(x) \ &= \ \ol{\{ f^n(x) : n \in \Z \}} \quad \text{for } x \in X, \\
\NN f_{\pm} \ &= \ \NN f \cup 1_X \cup \NN f^{-1}, \\
\RR f_{\pm} \ &= \ \RR f \cup 1_X \cup \RR (f^{-1}).
\end{split}
\end{equation}
On the other hand, $\CC f_{\pm}$ is in general larger than $\CC f \cup 1_X \cup \CC f^{-1}$.
The latter relation need not be transitive. See Example \ref{ex2.06} below.

Since $1_X \subset \OO f_{\pm}$, every point is recurrent for  $f_{\pm}$.

An action of a group $G$  on $X$ is a homomorphism $\rho : G \to H(X)$. If $G$ is a subgroup of $H(X)$,
then $G$ acts on $X$ by evaluation. That is, $\rho$ is the inclusion.
With the action understood we let $h_G = \bigcup \{ f \in \rho(G) \}$.  This is just the orbit relation of the action of
 $G$ on $X$. It is an equivalence relation but usually not closed.  Since it is an equivalence
relation, $\OO h_G = h_G$ and $\NN h_G = \ol{h_G}$. Since $h_G$ is reflexive, every point is recurrent for $h_G$.

If $G$ is the cyclic group generated by a homeomorphism $f$, then $h_G = \OO f_{\pm}$.

When $G = H(X)$, we write $h_X$ for $h_G$.  Clearly, $H(X)$\hyp{}decomposability as described in the Introduction is
just $h_X$-decomposability.

For a space $X$, let $\Iso(X)$ denote the set of isolated points of $X$.

\begin{prop}\label{prop2.03} Let $f \in H(X)$ and $\rho : G \to H(X)$ be an action of a group $G$ on $X$.
Let $\BB$ be a countable base of nonempty open sets for
$X$, and let $\BB^*$ be the collection of finite covers of $X$ by elements of $\BB$.
\begin{itemize}
\item[(a)] The homeomorphism $f$ is central iff the $G_{\del}$ set of recurrent points, $|\om f|$, is dense.
\begin{equation}\label{2.04}
|\om f| \ = \  \bigcap_{\A \in \BB^*, \ n \in \N} \ \bigcup_{U \in \A, \ i \geq n} \{ U \cap f^{-i}(U) \}.
\end{equation}
If $f$ is central, then any isolated point $x$ of $X$ is a periodic point for $f$.

\item[(b)] The homeomorphism $f$ is topologically transitive iff the set of transitive points, $\Trans(f)$, is nonempty, in which case
\begin{align}\label{2.05}
\Trans(f) \ &= \ \{ x : \om f(x) \ = \ X \} \\
&= \ \bigcap_{U \in \BB, \ n \in \N} \ \bigcup_{i \geq n} \{  f^{-i}(U) \}
\end{align}
is a dense $G_{\del }$ subset of $X$. If $f$ is topologically transitive, then $f^{-1}$ is topologically transitive. If $f$ is
topologically transitive, then $X$ is perfect or consists of a single periodic orbit for $f$.

\item[(c)] The relation $f_{\pm}$ is topologically transitive iff the set of transitive points,
$\Trans(f_{\pm})$, is nonempty, in which case
\begin{equation}\label{2.06}
\Trans(f_{\pm}) \ = \ \bigcap_{U \in \BB} \ \bigcup_{ \ n \in \Z} \{  f^{-n}(U) \}
\end{equation}
is a dense $G_{\del}$ subset of $X$. If $f_{\pm}$ is topologically transitive and $x$ is an isolated point of $X$, then
\begin{equation}\label{2.06a}
\Trans(f_{\pm}) \ = \ \OO f_{\pm}(x) \ = \ \Iso(X).
\end{equation}
Thus, if $f_{\pm}$ is topologically transitive, then either $\Iso(X) = \emptyset$, so $X$ is perfect, or
else $\Iso(X)$ is dense. If $\Iso(X)$ is finite and nonempty, then $X $ consists of a single periodic orbit.

\item[(d)] The following are equivalent:
\begin{enumerate}
\item[(i)] $f$ is topologically transitive;

\item[(ii)] $f_{\pm}$ is topologically transitive and $f$ is central;

\item[(iii)] $f_{\pm}$ is topologically transitive and $X$ is either perfect or consists of a single periodic orbit for $f$.
\end{enumerate}
In that case, $\Trans(f_{\pm}) = \Trans(f) \cup \Trans(f^{-1})$.

\item[(e)] The relation $h_G$ is topologically transitive iff the set of transitive points, $\Trans(h_G)$, is nonempty, in which case
\begin{equation}\label{2.07}
\Trans(h_{G}) \ = \ \bigcap_{U \in \BB} \ \bigcup_{ \ f \in G} \{  f^{-1}(U) \}
\end{equation}
is a dense $G_{\del }$ subset of $X$.
\end{itemize}
\end{prop}

\begin{proof}
For a relation $R$ on $X$ and subsets $U,V \subset X$, let
\begin{equation}\label{2.08}
\begin{split}
N_R(U,V) \ &= \ \{ n \in \N : R^n(U) \cap V \not= \emptyset \} \\
&= \ \{ n \in \N : U \cap R^{-n}(V) \not= \emptyset \}.
\end{split}
\end{equation}
Clearly, $R$ is central iff $N_R(U,U) \not= \emptyset$ for all nonempty open subsets $U$, and $R$ is topologically transitive
iff $N_R(U,V) \not= \emptyset$ for all nonempty open subsets $U, V$. If $R = f \in H(X)$ and $N_f(U,V)$ is finite for
some open $U, V$, then with $n = \max N_f(U,V)$ we let $W = U \cap f^{-n}(V)$. Then $W$ is a nonempty open set and
$N_f(W,V) = \emptyset$. Hence, if $f$ is central, then $N_f(U,U)$ is infinite, and if $f$ is topologically transitive, then
$N_f(U,V)$ is infinite for all nonempty open $U,V$.

The equations (\ref{2.04})-(\ref{2.08}) are easy to check, and density follows from the Baire Category Theorem.

Now assume that $f_{\pm}$ is topologically transitive and that $x$ is an isolated  point of $X$. If $V$ is a nonempty open set, then
$N_{f_{\pm}}(\{ x \}, V) = \emptyset$ unless $V$ meets the  orbit $\OO f_{\pm}(x)$. Hence,
$x \in \Trans(f_{\pm})$. If $y$ is another isolated point, then $N_{f_{\pm}}(\{x \},\{ y \}) \not= \emptyset $ implies that
$y$ is in the $f_{\pm}$ orbit of $x$.  Thus, (\ref{2.06a}) holds. In particular, $\Iso(X)$ is dense if it is nonempty.
If it is finite and nonempty, then $X = \Iso(X)$ because the latter is closed and dense. Since the elements of $\Iso(X)$
lie on a single orbit, $X$ consists of a single periodic orbit.

%
If $f$ is central and $x \in \Iso(X)$, then $N_f(\{ x \},\{ x \}) \not= \emptyset$ iff $x$ is a periodic point.
Thus, if $f_{\pm}$ is transitive and $f$ is central, then either $X$ is perfect or it consists of a single periodic orbit,
proving the implication (ii) $\Rightarrow$ (iii) in (d).
Since (i) obviously implies
(ii) in (d), it follows that if $f$ is topologically transitive, then either $X$ is perfect or consists of a single
periodic orbit. So if  $\RR f(x) = X$, then $\ol{ \{ f^i(x) : i \geq n \} } \ = \ X$ for every $n \in \N$. Intersecting, we see that
$\om f(x) = X$.  Thus, $\Trans(f) = \{ x : \om f(x) = X \}$.

Finally, if $X$ is a single periodic orbit, then $f$ is topologically transitive and every point is a transitive point for $f$ and
$f^{-1}$.  Now assume that $X$ is perfect and that $f_{\pm}$ is transitive. If $x \in \Trans(f_{\pm})$, then
$\ol{ \{ f^i(x) : i \in \Z, |i| \geq n \} } \ = \ X$ for all $n \in \N$. Intersecting, we obtain that $\om f(x) \cup \al f(x) = X$.
In particular, $x$ is in one of these. If $x$ is contained in a closed invariant set $A$ like $\om f(x)$ or $\al f(x)$, then
$X = \RR (f_{\pm})(x) \subset A$. Thus, $x \in \Trans(f) \cup \Trans(f^{-1})$, so either $f$ or $f^{-1}$ is topologically
transitive.  But $\NN (f^{-1}) = (\NN f)^{-1}$ implies that $f^{-1}$ is topologically transitive if $f$ is.
\end{proof}

\begin{remark} There are various, slightly conflicting, definitions of topological transitivity. These are sorted out
in \cite{A-C}. We are following \cite{A-93}.
\end{remark}

\begin{lem}\label{lem2.04} (a) Let $f \in H(X)$ and $x \in X$. The sets $\om f(x)$ and $\al f(x)$ are $f$-invariant.
If $x \in |\CC f|$, then $\CC f(x)$ and $\CC f^{-1}(x)$ are $f$-invariant.

(b) The chain components of $1_X$ are the components of $X$.
\end{lem}

\begin{proof}
(a) For a bijection $f$ on $X$, a subset $A$ is $f$-invariant iff it is $f^{-1}$-invariant
iff it is forward invariant for both $f$ and $f^{-1}$.

When $f$ is a continuous map on $X$, then
 $y \in \om f(x)$ when there is a subsequence $\{ f^{n_i}(x) \}$ of the
 orbit sequence which converges to $y$.  Then $\{ f^{n_i+1}(x) \}$ converges to $f(y)$, and if a subsequence of
 $\{ f^{n_i - 1}(x) \}$ converges to $z$, then $f(z) = y$.  That is, $\om f(x)$ is a nonempty, closed, $f$-invariant
 subset of $X$ when $f$ is a continuous map on $X$. So when $f$ is a homeomorphism, the same is true for
 $\al f(x)$.

 For $f \in H(X)$ we obtain from the identity (\ref{2.02}) that $1_X \cup \CC f = f^{-1} \circ \CC f$. If $x \in \CC f(x)$,
then $\CC f(x) \ = \ \{ x \} \cup \CC f(x) \ = \ f^{-1}(\CC f(x))$.  That is, $\CC f(x)$ is $f^{-1}$
invariant and so is $f$ invariant. Since $\CC (f^{-1}) = (\CC f)^{-1}$ the same is true for $\CC f^{-1}(x)$.

(b) The space of chain components being zero-dimensional by Proposition \ref{prop2.01}(d), every connected set contained
in $|\CC R|$ is entirely contained in a single chain component. In particular, when $R$ is $1_X$, every component is
a subset of a chain component.  On the other hand, when $R = 1_X$, every clopen subset is inward for $R$ and so contains any chain
component that it meets. It follows that the components are the chain components.
\end{proof}

\begin{prop}\label{prop2.05} Let $f \in H(X)$ and $\rho :G \to H(X)$ be an action of a group $G$ on $X$.
\begin{itemize}
\item[(a)] The following are equivalent:
\begin{itemize}
\item[(i)] $X$ is $f$-indecomposable;
\item[(ii)] $X$ is $f_{\pm}$-indecomposable;
\item[(iii)] $f_{\pm}$ is chain transitive.
\end{itemize}

\item[(b)] The relation $h_G$ is chain transitive iff $h_G$ is indecomposable.

\item[(c)] If $f$ is chain recurrent, then $\CC f = \CC (f_{\pm})$.

\item[(d)] The following are equivalent:
\begin{itemize}
\item[(i)] $f$ is chain transitive;
\item[(ii)] $f$ is chain recurrent and $f_{\pm}$ is chain transitive;
\item[(iii)] $f$ is chain recurrent and indecomposable.
\end{itemize}

\item[(e)] If $X$ is connected, then $f$ is chain transitive iff it is chain recurrent. In particular, $1_X$ is chain transitive.
\end{itemize}
\end{prop}

\begin{proof}
(a), (b) Since $h_G = (h_G)_{\pm}$, these results are a special cases of Proposition \ref{prop2.02}
applied with $R = f$ and $R = h_G$.

(c) If $f$ is chain recurrent, then $|\CC f| = X$ and $1_X \subset \CC f$. For any $x \in X$, we have $x \in \CC f(x)$, and this set
is $f$-invariant by Lemma \ref{lem2.04}. Hence, $f^{-1}(x) \in \CC f(x)$. Thus, $f^{-1} \subset \CC f$. Since
$f \subset f_{\pm} \subset \CC f$, it follows from (\ref{2.01}) that $\CC f \subset \CC (f_{\pm}) \subset \CC \CC f = \CC f$.

(d) If $f$ is chain transitive, then it is clearly chain recurrent, and $f \subset f_{\pm}$ implies that $f_{\pm}$ is chain
transitive.  The converse follows from  (c). This proves the equivalence of (i) and (ii). The equivalence of (ii) and (iii) follows from
(a).

(e) If $X$ is connected, then by Lemma \ref{2.04}(b)
$X$ consists of a single chain component for $1_X$, and so $1_X$ is chain transitive. So if $f$ is chain recurrent, then
$X \times X = \CC 1_X \subset \CC \CC f = \CC f$ by (\ref{2.01}) again.
\end{proof}

\begin{ex}\label{ex2.06} On $\Z$ let $t$ be the translation bijection given by $n \mapsto n+1$. Let
$\Z^*$ be the one-point compactification adjoining the point $\pm \infty$ to $\Z$, and let
$\Z^{**}$ be the two-point compactification adjoining the points $- \infty, + \infty$ to $\Z$.
Let $t^*$ and $t^{**}$ be the homeomorphisms extending $t$ to $\Z^*$ and $\Z^{**}$, respectively.
Both $t^*_{\pm}$ and $t^{**}_{\pm}$ are topologically transitive with $\Z$ the orbit of isolated points,
and so $\Z = \Trans(t^{*}_{\pm}) = \Trans(t^{**}_{\pm})$. Of course, both $t^{*}_{\pm}$ and $t^{**}_{\pm}$ are
chain transitive, but $t^*$ is also chain transitive while $t^{**}$ is not.

Let $X$ be the quotient space of $Z^{**} \times \{ 0,1 \}$ with the fixed points $(+\infty, 0)$ and$(+\infty,1)$
identified. Let $f$ be the homeomorphism on $X$ induced by $t^{**} \times 1_{\{ 0,1 \}}$. Clearly, $f_{\pm}$ is
chain transitive. But $\CC f \cup 1_X \cup \CC f^{-1}$ is contained in the image of
$(Z^{**} \times \{ 0 \} )^2 \cup (Z^{**} \times \{ 1 \} )^2 $ in $X^2$ and so is a proper subset of $\CC (f_{\pm}) = X^2$.
Furthermore, it is easy to check that in this case $\CC (f \cup 1_X) \ = \ (\CC f) \cup 1_X$. Hence, $R = f \cup 1_X$
is a closed relation such that $R$ is chain recurrent and $R_{\pm}$ is chain transitive, but $R$ is not chain transitive.
Thus, Proposition \ref{prop2.05} (c) and (d) do not extend to general relations.
\end{ex}

Finally, we recall the \emph{Uniqueness of Cantor}: any compact, perfect, zero-dimensional, metrizable space is homeomorphic
to the Cantor set in $[0,1]$. We will call any such space a Cantor set.  We will need the following well-known
result, providing a brief sketch of the proof.

\begin{prop}\label{prop2.06} For any space $X$ there exists a surjective continuous map
$\pi : C \to X$ with $C$ a Cantor set. \end{prop}

\begin{proof}
With $\BB$ a countable basis for $X$, let $Z$ be the closure in $X \times \{ 0, 1 \}^{\BB}$ of
the set of pairs $\{ (x,z) : z_U = 1 \ \Leftrightarrow \ x \in U \}$. The projection to $X$ is clearly onto, and
because $\BB$ is a basis, the projection to $\{ 0 , 1 \}^{\BB}$ is easily seen to be injective. It follows that
$Z$ is compact and zero-dimensional.  If $C_0$ is a Cantor set, then $C = Z \times C_0$ is perfect as well
as zero-dimensional and so is a Cantor set. Let $\pi$ be the composition of projections $C \to Z \to X$.
\end{proof}

\section{Spaces which admit Chain Transitive Maps}\label{trans}

We begin with the relationship between omega limit sets and chain transitive subsets which was described in the introduction.

\begin{prop}\label{prop3.01} If $f$ is a homeomorphism on a space $X$ and  $x \in X$, then
$\om f(x)$ and $\al f(x)$ are chain transitive subsets, i.e. the restriction of $f$ to each of these
nonempty, closed, invariant sets is chain transitive. \end{prop}

\begin{proof}
Let $y, y' \in \om f(x)$ and let $\ep > 0$. Choose $\del > 0$ an $\ep/2$-modulus of uniform continuity for
$f$ with $\del <  \ep/2$. There exists $N \in \N$ such that $n \geq N$ implies $f^n(x) \in \bar V_{\del }(\om f(x))$. There exist
$n \geq N$ such that $d(f^n(x),y) < \del $ and $k \in \N$ such that $d(f^{n+k}(x),y') < \del $. For $i = 0,\dots,k$
choose $y_i \in \om f(x)$ such that $d(f^{n+i}(x),y_i) \leq \del $ with $y_0 = y, y_k = y'$. Hence,
$d(f^{n+i+1}(x),f(y_i)) \leq \ep/2$ and so $d(y_{i+1},f(y_i)) \leq \ep$. That is, $\{ y_i \}$ is an $\ep$-chain from
$y$ to $y'$. It follows that $\om f(x)$ is a chain transitive subset. Applying the result to $f^{-1}$, we see that
$\al f(x) = \om (f^{-1})(x)$ is a chain transitive subset as well.
\end{proof}

We prove, conversely, in Theorem \ref{theo3.11} (a) below that a chain transitive homeomorphism is the restriction
to an omega limit set of a homeomorphism in a larger system.

In the constructions which follow, we will repeatedly use the process of \emph{attachment}. Assume that
$A$ is a nonempty, closed, nowhere dense subset of a space $X$ and that $h : A \to B$ is a continuous surjection.
We may assume that $X$ and
$B$ are disjoint. Otherwise, replace $X$ by $X \times \{ 0 \}$ and $B$ by $B \times \{ 1 \}$.
Define $X/h$, the space with \emph{$X$ attached to $B$ via $h$} as follows:  Let $\wt h$ denote the continuous
retraction $h \cup 1_B : A \cup B \to B$. Let
$E_h = 1_X \cup (\wt h^{-1} \circ \wt h) = 1_X \cup (\wt h \times \wt h)^{-1}(1_{B})$, a closed
equivalence relation on $X \cup B$. Let $X/h$ be the quotient space with projection $q_h : X \cup B \to X/h$. Since $q_h$ is
injective on $B$, we may regard it as an identification and so regard $B$ as a subset of $X/h$. Furthermore, $q_h$
restricts to a surjection $X \to X/h$ which maps $A$ onto $B$ via $h$ and which is a homeomorphism between the dense open sets
$X \setminus A \subset X$ and $(X/h) \setminus B \subset X/h$.
When $B$ is a singleton, we write $q_A : X \to X/A$ for the quotient map and describe the result as \emph{smashing $A$ to a point}.

We will use some results of E. R. Lorch from \cite{L1} and \cite{L2} (see also \cite{T}), which we will briefly review.

For a locally compact space $W$, a \emph{compactification} of $W$ is a compact space $Y$ together with a
dense embedding of $W$ into $Y$,
i.e. a homeomorphism of $W$ onto a dense subset of $Y$. Because $W$ is locally compact, its image is an open, dense subset
of $Y$, so $X = Y \setminus W$ is a nowhere dense, closed subset of $Y$. Reversing the point of view, we call $Y$ an
\emph{extension of $X$} if $Y$ is a compact space and $X$ is a closed, nowhere dense subset of $Y$.

By a \emph{pair of spaces} $(Y, X)$ we will mean a space and a closed subset, respectively. Recall our default assumption that a space is a
nonempty, compact, metrizable space. We will call $(Y, X)$ an
\emph{extension pair} when $Y$ is an extension of $X$, i.e. when $X$ is nowhere dense
in $Y$.

A continuous map $f : Y_1 \to Y_2$ is a map of pairs $f : (Y_1, X_1) \to (Y_2, X_2)$ when
$f(X_1) \subset X_2$. If $X_1 = X_2 = X$, we will say that $f : (Y_1, X) \to (Y_2, X)$ is a map \emph{rel $X$} if it restricts
to the identity on $X$ and if, in addition, $f(Y_1 \setminus X) \subset Y_2 \setminus X$. These conditions imply that
$1_X = (f \times f)^{-1}(1_X) $.  So by intersecting over $\del > 0$ and using compactness, we see that for every $\ep > 0$
there exists $\del > 0$ such that,
\begin{equation}\label{3.01}
\begin{gathered}
\text{if } (u,v) \in Y_1 \times Y_1 \text{ and} \\
d_2(f(u),f(v)), d_2(f(u),X), d_2(f(v),X) \leq \ \del, \text{ then} \\
(u,v) \in \bar V^{d_1}_{\ep}, \text{ and so } f^{-1}(\bar V^{d_2}_{\del}(x)) \ \subset \ \bar V^{d_1}_{\ep}(x) \text{ for all } x \in X,
 \end{gathered}
 \end{equation}
where $d_1, d_2$ are metrics on $Y_1$ and $Y_2$. Note  that in considering different extensions $Y_1, Y_2$ of the same
space $X$ we do not assume that the metrics $d_1$ and $d_2$ agree on $X$, although they are, of course, uniformly
equivalent on $X$.

\begin{df}\label{df3.02}  We call
 $X^i$  an \emph{isolated point extension} of $X$, or just a \emph{point extension} of $X$, if
$X^i$ is an extension of $X$ with each point of $X^i \setminus X$ isolated.  We then call
$(X^i, X)$ a \emph{point extension pair}.
\end{df}

A pair $(Y,A)$ is a point extension pair iff $Y$ is infinite, $\Iso(Y)$ is dense and $A = Y \setminus \Iso(Y)$.
Since $Y$ is separable, $\Iso(Y)$ is denumerable, and so Lorch uses the term \emph{denumerable extension} instead of
point extension. Thus, $Y$ is a compactification of the denumerable discrete set $\Iso(Y)$.

For a point extension $(X^i,X)$, we define a \emph{canonical retraction} $r : X^i \to X$ so that
for all $x \in X^i$
\begin{equation}\label{3.02}
d(x, r(x)) \ = \ d(x, X).
\end{equation}
The choice depends on the metric. Even for a fixed metric there may be more than one choice of point $r(x)$ which
satisfies (\ref{3.02}). For each $x$ we fix a choice to define $r$. Clearly,
\begin{equation}\label{3.03}
d(r(x_1),r(x_2)) \ \leq \ d(x_1,X) + d(x_1,x_2) + d(x_2,X).
\end{equation}
For every $\ep > 0$ there are only finitely many points $x \in X^i$ with
$d(x,X) \geq \ep$, so continuity of $r$ at the points of $X$ follows.  Continuity at the isolated
points is trivial.

Notice that if $N_0$ is a cofinite subset of $X^i \setminus X$, then the closure of $N_0$ contains $X$, so
$r(N_0)$ is dense in $X$.

We now recall the elegant proof of \emph{Lorch's Uniqueness Theorem}, \cite{L1} Proposition 10.

\begin{theo}\label{theo3.03} Every space $X$ has an essentially unique isolated point extension. That is, if $(Y, X)$ and $(Y', X)$
are point extension pairs, then there is a homeomorphism $f : (Y, X) \to (Y', X) $ rel  $ X$. \end{theo}

\begin{proof}
Let $\{ x_n : n \in \N \}$ be a sequence of not necessarily distinct points in $X$
such that $\{ x_n : n \geq N \}$ is
dense in $X$ for all $N \in \N$.  Let $X^i = X \times \{ 0 \} \cup \{ (x_n, n^{-1}) : n \in \N \}$, and
identify $X$ with $X \times \{ 0 \}$. Clearly, $(X^i, X)$ is a point extension pair, so $X$ has at least one
point extension.

Let $(Y, X)$ and $(Y', X)$ be point extension pairs.

Fix metrics $d$ and $d'$ on $Y$ and $Y'$, respectively. Define a metric $d''$ on $X$ as the pointwise maximum of $d$ and $d'$.  The three metrics
$d$, $d'$ and $d''$ are uniformly equivalent on $X$. Let $r : Y \to X$ and $r' : Y' \to X$ be canonical retractions.
Let $N = \Iso(Y) = Y \setminus X$ and $N' = \Iso(Y') = Y' \setminus X$.
Use a counting of $N$ and $N'$ to impose orderings which are order isomorphic to $\N$ and so are well-orderings.

Let $a_1$ be the first element of $N$, and choose $b_1$ to be the first element of $N'$ which satisfies
\begin{equation}\label{3.04}
d''(r(a_1), r'(b_1)) \ < \ d(a_1, X).
\end{equation}

Let $b_2$ be the first element of $N' \setminus \{ b_1 \}$, and choose $a_2$ to be the first element of $N \setminus \{ a_1 \}$ such that
\begin{equation}\label{3.05}
d''(r(a_2), r'(b_2)) \ < \ d'(b_2, X).
\end{equation}

Proceed inductively.  If $n$ is even, let $a_{n+1}$ be the first element of $N \setminus \{ a_1, \dots, a_n \}$, and
if $n$ is odd, let $b_{n+1}$ be the first element of $N' \setminus \{ b_1, \dots, b_n \}$. We can then choose $b_{n+1}$ or
$a_{n+1}$ so that
\begin{equation}\label{3.06}
d''(r(a_{n+1}), r'(b_{n+1})) \ < \ \max\{d(a_{n+1}, X), d'(b_{n+1}, X)\}.
\end{equation}

By construction, $\{ a_n \}$ and $\{ b_n \}$ are re-numberings of the sets $N$ and $N'$.
Define the mapping $f : Y \to Y'$ as an extension of the identity on $X$ by putting
$f(a_n) = b_n$  for all $n$.  Let $m_n = \max \{d(a_{n}, X), d'(b_{n}, X)\}$. Observe that $m_n \longrightarrow 0$ as
$n \longrightarrow \infty$.

Continuity of $f$ is clear at the isolated points. Suppose $x \in X$ and
$a_{n_i} \longrightarrow x$ so that $n_i \longrightarrow \infty$.
\begin{equation}\label{3.07}
\begin{split}
d(r(a_{n_i}),x) \ &\leq \ d(a_{n_i}, x) \ + \ d(r(a_{n_i}), a_{n_i}) \\
&\leq \ d(a_{n_i}, x) \ + \ m_{n_i} \ \longrightarrow \ 0,
\end{split}
\end{equation}
and so $d'(r(a_{n_i}),x) \longrightarrow 0$.
Hence,
\begin{equation}\label{3.08}
\begin{split}
d'(b_{n_i},x) \ &\leq \ d'(r(b_{n_i}), b_{n_i}) \ + \ d'(r(b_{n_i}),r(a_{n_i})) \ + \ d'(r(a_{n_i}),x) \\
&\leq \ 2 m_{n_i} \ +  \ d'(r(a_{n_i}),x) \ \longrightarrow \ 0.
\end{split}
\end{equation}

Continuity of $f$ follows. The result for $f^{-1}$ is similar and also follows from compactness.
\end{proof}

\begin{cor}\label{cor3.04} Let $(X^i, X)$ and $(Y^i, Y)$ be point extension pairs and $h : X \to Y$ a surjective
continuous map.  There exist a continuous map $H : (X^i, X) \to (Y^i, Y)$ which restricts to $h$ on $X$ and
to a homeomorphism of $\Iso(X^i) = X^i \setminus X$ onto $\Iso(Y^i) = Y^i \setminus Y$. In particular, if
$h$ is a homeomorphism, so is $H$. \end{cor}

\begin{proof}
We  attach $Y$ to $X^i$ using $h$, letting $q_h : X^i \to X^i/h$
be the quotient map. We regard $Y$ as a subset of $X^i/h$
so that $(X^i/h,Y)$ is a point extension pair and $q_h : X^i \to X^i/h$ is an extension of $h$ which is a homeomorphism from
$X^i \setminus X$ onto $X^i/h \setminus Y$.  By Theorem \ref{theo3.03} there is a homeomorphism $f :(X^i/h, Y) \to (Y^i, Y)$
rel $Y$. Let $H = f\circ q_h$.
\end{proof}

\begin{lem}\label{lem3.05} Let $(X_1^i, X_1)$ and $(X_2^i, X_2)$ be point
extension pairs, and let $H : (X_1^i, X_1) \to (X_2^i, X_2)$
be a continuous map of pairs. Let $(Y_1, X_1)$ and $(Y_2, X_2)$ be extension pairs with
$\pi_1 : (Y_1, X_1) \to (X^i_1, X_1)$ and $\pi_2 : (Y_2, X_2) \to (X^i_2, X_2)$  pair maps rel $X_1$ and rel $X_2$,
respectively. Assume that $\wt H : Y_1 \to Y_2$ is a function such that $\pi_2 \circ \wt H = H \circ \pi_1$,
i.e. the following diagram commutes:
\begin{equation*}
\begin{tikzcd}
Y_1 \arrow{r}{\wt H} \arrow{d}[swap]{\pi_1} & Y_2 \arrow{d}{\pi_2}\\
X_1^i \arrow{r}[swap]{H} & X_2^i
\end{tikzcd}
\end{equation*}
If for each $x \in \Iso(X^i_1)$ the restriction $\wt H : \pi_1^{-1}(x) \to \pi_2^{-1}(H(x))$ is continuous, then
$\wt H$ is continuous on $Y_1$. \end{lem}

\begin{proof}
If $x \in \Iso(X^i_1)$, then $\pi_1^{-1}(x)$ is a clopen set, and so
$\wt H$ is continuous at the points of $\pi_1^{-1}(\Iso(X^i))$ by hypothesis.

Let $x \in X_1$, so $y = H(x) \in X_2$. Given $\ep > 0$, there exists by (\ref{3.01}) $\del > 0$ so that
$\pi_2^{-1}(V^{d_2}_{\del }(y)) \subset V^{d_2'}_{\ep}(y)$, where $d_2$ and $d_2'$
are the metrics on $X^i_2$ and $Y_2$, respectively. By continuity of $H$ there exists $\ga > 0$ so
that $H(V^{d_1}_{\ga }(x)) \subset V^{d_2}_{\del }(y)$.  If $x_1 \in \pi_1^{-1}(V^{d_1}_{\ga }(x))$, then
$\wt H(x_1) \in \pi_2^{-1}(V^{d_2}_{\del }(y)) \subset V^{d_2'}_{\ep}(y)$. Hence, $\pi_1^{-1}(V^{d_1}_{\ga }(x))$
is a neighborhood of $x$ in $Y_1$ which is mapped by $\wt H$ into the neighborhood $V^{d_2'}_{\ep}(y)$ in $Y_2$, and
continuity at $x$ follows.
\end{proof}

\begin{df}\label{df3.06}  For a space $K$ and a pair of spaces $(Y,X)$, we call $Y$
a \emph{$K$-extension} of $X$ if there exist a point extension pair $(X^i, X)$ and a
map of pairs $\pi : (Y, X) \to (X^i, X)$ rel $X$ such that $\pi^{-1}(x)$ is
homeomorphic to $K$ for every $x \in \Iso(X^i)$.  We then call
$(Y, X)$ a \emph{$K$-extension pair}, and the space $Y$ is denoted by $X^{(K)}$. \end{df}

We extend Lorch's Theorem.

\begin{theo}\label{theo3.07} For any given space $K$, every space $X$ has an essentially
unique $K$-extension pair $(X^{(K)},X)$. Furthermore, if $(X^{(K)}, X)$ and $(Y^{(K)}, Y)$
are $K$-extension pairs and $h : X \to Y$ is a surjective
continuous map, then there exists a continuous map $H : (X^{(K)}, X) \to (Y^{(K)}, Y)$ which restricts to $h$ on $X$ and
to a homeomorphism of $ X^{(K)} \setminus X$ onto $Y^{(K)} \setminus Y$. In particular, if
$h$ is a homeomorphism, then so is $H$. \end{theo}

\begin{proof}
Let $(X^i, X)$ be a point extension of $X$. Let $\pi :(X^i \times K,  X \times K) \to (X^i, X)$
be the map of pairs given by the first coordinate
projection. Attach $X^i \times K$ to $X$ by using $h = \pi|_{X \times K}$. The map $\pi$ factors through the quotient map $q_h$
to define a map $( (X^i \times K)/h, X) \to (X^i, X)$ rel $X$. Thus, $((X^i \times K)/h, X)$ is a $K$-extension pair.

Now let $\pi_X : (X^{(K)}, X) \to (X^i, X)$ and $\pi_Y : (Y^{(K)}, Y) \to (Y^i, Y)$ be
maps rel $X$ and $Y$, respectively, with each
fiber over $X^i \setminus X$ and $Y^i \setminus Y$ homeomorphic to $K$. Use Corollary \ref{cor3.04} to get
an extension $H^i : (X^i, X) \to (Y^i, Y)$ which maps $X^i \setminus X$ to $Y^i \setminus Y$ homeomorphically.
Define $H$ by choosing for each
$x \in X^i \setminus X$  an arbitrary homeomorphism from $\pi_X^{-1}(x)$ to $\pi_Y^{-1}(H^i(x))$. These exist
because each fiber is homeomorphic to $K$. By Lemma \ref{lem3.05}, the resulting $H$ is continuous.

In particular, if $Y = X$ and $h = 1_X$, then it follows that the $K$-extension is essentially unique.
\end{proof}

Two cases are of special interest to us. Recall that a component of a space $X$ is an \emph{isolated component} if it
is a clopen subset of $X$.

\begin{theo}\label{theo3.08} Let $(Y, X)$ be an extension pair.
\begin{itemize}
\item[(a)] If $K$ is a Cantor set, then $(Y, X)$ is a $K$-extension pair iff  $Y$
is perfect and the dense, open set $Y \setminus X$ is zero-dimensional.
\item[(b)] If $K$ is a connected space, then $(Y, X)$ is a $K$-extension pair iff \begin{itemize}
\item The union of the isolated components is dense in $Y$.
\item Each isolated component is homeomorphic to $K$.
\item \emph{(diameter condition)} For every $\ep > 0$ there are only finitely many isolated components with diameter greater than $\ep$.
\end{itemize}
\end{itemize}
\end{theo}

\begin{proof}
Let $\{ A_n \}$ be a pairwise disjoint sequence of nonempty clopen subsets of $Y$ with
union $Y \setminus X$. If for every $\ep > 0$ only finitely many of the sets $A_n$ have diameter greater than $\ep$, then
\begin{equation}\label{3.08a}
E \ = \ 1_X \ \cup \ \bigcup_n \{ A_n \times A_n \}
\end{equation}
 is a closed equivalence relation.
If $q : Y \to Y/E$ is the quotient space projection, then $(Y/E, X)$ is a point extension pair, $q$ is a map of
pairs rel $X$, and
the fibers of $q$ over the isolated points are the sets $A_n$.   Conversely, if $\pi : (Y, X) \to (X^i, X)$ is a
map of pairs rel $X$, then by (\ref{3.01}) there are only finitely many fibers $\pi^{-1}(x)$ with diameter at least $\ep$.

(a) If $\pi : (Y, X) \to (X^i, X)$ is a map rel $X$ with each fiber over a point of $X^i \setminus X$ a Cantor set,
then as a countable disjoint union of Cantor sets, $Y \setminus X$ is zero-dimensional and $\Iso(Y) = \emptyset$.

Conversely, if the locally compact space $Y \setminus X$ is  zero-dimensional with no isolated points, then we can express it
as the union of a pairwise disjoint sequence $\{ C_n : n \in \N \}$ of nonempty, clopen subsets of $Y$ each of which is thus a
Cantor set. Let $\{ C_{n,i} : i = 1, \dots, N_n \}$ be a partition of $C_n$ by nonempty clopen subsets of diameter
less than $n^{-1}$.  If $\{ A_n \}$ is a counting of the collection $\{ C_{n,i} : n \in \N, i = 1, \dots, N_n \}$,
then  with $E$ as in (\ref{3.08a})  the projection
$q : (Y, X)  \to (Y/E, X)$ is a map rel $X$ with each fiber over a point of $(Y/E) \setminus X$
a Cantor set.  Thus, $(Y, X)$ is a Cantor set extension pair.

(b) If $\pi : (Y, X) \to (X^i, X)$ is a map rel $X$ with each fiber connected, then $\{ \pi^{-1}(x) : x \in X^i \setminus X \}$
is the set of isolated components, so the conditions of (b) are necessary.

Conversely, if they hold, then
we let $\{ A_n \}$ be the sequence of isolated components. There are infinitely many isolated components
because $X$ is nonempty and $Y \setminus X$ is dense. With $E$ as in (\ref{3.08a}) again,
 $q : (Y, X) \to (Y/E, X)$ is the required map rel $X$.
\end{proof}

\begin{cor}\label{cor3.09} If $Y$ and $X$ are Cantor sets with closed, nowhere dense subsets $Y_1 \subset Y$ and $X_1 \subset X$,
then for any surjective continuous map $h : Y_1 \to X_1$ there is a continuous map $H : Y \to X$ which extends
$h$ and which restricts to a homeomorphism of $Y \setminus Y_1$ to $X \setminus X_1$.  In particular, if $h$ is
a homeomorphism, then so is $H$. \end{cor}

\begin{proof}
By Theorem \ref{theo3.08} (a), $(Y, Y_1)$ and $(X, X_1)$ are Cantor set extension pairs. The existence
of $H$ then follows from Theorem \ref{theo3.07}.
\end{proof}

\begin{remark}
This is a classical
theorem of Knaster and  Reichbach, extended by Gutek, see \cite{K-R} and \cite{G}.
\end{remark}

Now we apply these results.

\begin{lem}\label{lem3.10} Let $(X^i,X)$ be a point extension pair, $(Y, X)$ an extension pair and $\pi : (Y, X) \to (X^i, X)$ a surjective map of pairs rel $X$.
\begin{itemize}
\item[(a)] The map $\pi : Y \to X^i$ is open.

\item[(b)] Assume that $H : (Y, X) \to (Y, X)$ and $H^i : (X^i, X) \to (X^i, X)$ are homeomorphisms with $\pi $ mapping $H$ to $H^i$,
i.e. $\pi \circ H = H^i \circ \pi$.  If $H^i$ is chain transitive, then
$H$ is chain transitive.
\end{itemize}
\end{lem}

\begin{proof}
(a) If $U \subset X$ is open and $x \in \pi(U) \cap (X^i \setminus X)$, then $\pi(U)$ is
a neighborhood of $x$ because $x$ is an isolated point.

If $x \in \pi(U) \cap X$, then there exist $\ep > 0$ and $\del > 0$ so that $\bar V^{d_Y}_{\ep}(x) \subset U$ and $\pi^{-1}(\bar V^{d_{X^i}}_{\del }(x)) \subset \bar V^{d_Y}_{\ep}(x)$.  Since $\pi$ is surjective,
$\pi(\bar V^{d_Y}_{\ep}(x)) \ \supset \  \bar V^{d_{X^i}}_{\del }(x)$, so $\pi(U)$ is a neighborhood of $x$.

(b) Since $\pi$ is open, Proposition \ref{prop2.02b} (e) implies that $\pi$ is a semi-conjugacy from
$\CC H$ to $\CC H^i$ and from $\CC H^{-1}$ to $\CC (H^i)^{-1}$. Fix $x \in X$. For any $y \in Y$
we have $x \in \CC H^i (\pi(y))$ and $x \in \CC (H^i)^{-1}(\pi(y))$ because $H^i$ is chain transitive.
Since $\{ x \} = \pi^{-1}(x)$, it follows from the semi-conjugacy that $x \in \CC H (y)$ and $x \in \CC H^{-1}(y)$.
That is, every point of $Y$ is chain equivalent to $x$, and so transitivity of $\CC H$ implies that $H$ is chain
transitive.
\end{proof}

\begin{theo}\label{theo3.11} Let $(X^i, X)$ be a point extension pair, $(X^c, X)$ a Cantor set extension pair and $(X^{(K)}, X)$ a $K$-extension pair for spaces $X$ and $K$.
\begin{itemize}
\item[(a)] If $f$ is a chain transitive homeomorphism on $X$, then there exists
a homeomorphism $F$ on $X^i$ which extends $f$ on $X$ such that
\begin{itemize}
\item $F$ is chain transitive,
\item $F_{\pm}$ is topologically transitive, and
\item if $x \in X^i \setminus X$, then $\om F(x) = X = \al F(x)$.
\end{itemize}
\item[(b)] There exists $G^c$ a topologically transitive homeomorphism on $X^c$ which extends $1_X$.
\item[(c)] There exists $G^K$ a chain transitive homeomorphism on $X^{(K)}$ which extends $1_X$.
\end{itemize}
\end{theo}

\begin{proof}
(a) By concatenating $\ep$-chains which are
$\ep$-dense in $X$, we can obtain an infinite sequence $\{ x_k : k \in \Z \}$ with $d(f(x_k),x_{k+1}) \longrightarrow 0$ as
$|k| \longrightarrow \infty$ and so that for any $N \in \N$ the tails $\{ x_k : k \geq N \}$ and $\{ x_{-k} : k \geq N \}$
are dense in $X$.  Define the sequence $\{ y_k \in X \times [0,1] : k \in \Z\} $ by
\begin{equation}\label{3.09}
y_k \ = \
\begin{cases}
(x_k, (2k + 1)^{-1} ) & \text{for } k \geq 0, \\
(x_k, (2|k|)^{-1} ) & \text{for } k < 0.
\end{cases}
\end{equation}
Let $Y = X \times \{ 0 \} \ \cup \ \{ y_k : k \in \Z \}$, and define $\bar F(x,0) = (f(x),0)$ and $\bar F(y_k) = y_{k+1}$ for
$k \in \Z$. It is easy to see that $Y$ is an isolated point extension of $X = X \times \{ 0 \}$,
$\bar F$ is a homeomorphism on $Y$, and $X \times \{ 0 \} = \om \bar F(y_k) = \al \bar F(y_k)$ for any $k \in \Z$. Since the orbit
$\OO (\bar F_{\pm})(y_0) = \{ y_k : k \in \Z\}$ is dense, it follows that $\bar F_{\pm}$ is topologically transitive.
The homeomorphism $\bar F$ is chain transitive on $Y$ because $f$ is chain transitive on $X$ and because
$X = \om \bar F(y_k) \subset \CC \bar F(y_k)$ and also $X = \al \bar F(y_k) \subset \CC \bar F^{-1}(y_k)$.

By Lorch's Uniqueness Theorem \ref{theo3.03}
there exists a homeomorphism $H : (X^i, X) \to (Y, X)$ rel $X$. Let $F = H^{-1} \circ \bar F \circ H$.

(b)  We begin with $G$ a topologically transitive homeomorphism on a Cantor set  with a
Cantor set $C$ of fixed points. By topological transitivity, $C$ is necessarily nowhere dense. To be specific, let
 $G$ be the shift homeomorphism
on the product space $C^{\Z}$. Let $c : C \to C^{\Z}$ be the embedding with $c(x)_i = x$ for all $i \in \Z$. Thus, $c$
is a homeomorphism onto the set of fixed points. Since $C$ is nowhere dense, $(C^{\Z}, C)$ is a Cantor set extension pair
by Theorem \ref{theo3.08} (a).

For an arbitrary space $X$, there exists by Proposition \ref{prop2.06} a continuous surjection $h : C \to X$.
We attach $C^{\Z}$ to $X$ using $h$. Let $Y = (C^{\Z})/h$.
Now, $(Y, X)$ is a Cantor set extension pair, and $G$ factors to define a topologically transitive homeomorphism $\bar G$ which
restricts to the identity on $X$.

By Theorem \ref{theo3.07} there exists a homeomorphism $H : (X^c, X) \to (Y, X)$ rel $X$. Let $G^c = H^{-1} \circ \bar G \circ H$.

(c) First we consider the case where $X$ is the usual Cantor set $C$ in $[0,1]$ with $0, 1 \in C$. Then the complement
consists of a pairwise disjoint, countable collection $\{ (a_i,b_i) : i \in \N \}$ of open intervals in $(0,1)$.
Let $\ell_i = b_i - a_i$, so $\ell_i > 0$ for all $i \in \N$, but for any $\ep > 0$ there are only finitely many with $\ell_i \geq \ep$.
Let
\begin{equation}\label{3.10}
C^i = C \times \{ 0 \} \ \cup \ \bigcup_{i,m \in \N} \{ (a_i,\ell_i/m), (b_i,\ell_i/m) \}.
\end{equation}
Since the set of endpoints is dense in $C$, it follows that $C^{i}$ is a  point extension of $C = C \times \{ 0 \}$.
Now we relabel the isolated points. For $i \in \N, k \in \Z$, define
\begin{equation}\label{3.11}
\begin{split}
u_{i,k} \ = \
\begin{cases}
(a_i, \ell_i \cdot (2k + 1)^{-1} ) & \text{for } k \geq 0, \\
(b_i,  \ell_i \cdot (2|k|)^{-1} ) & \text{for } k < 0,
\end{cases} \\
v_{i,k} \ = \
\begin{cases}
(b_i, \ell_i \cdot (2k + 1)^{-1} ) & \text{for } k \geq 0, \\
(a_i,  \ell_i \cdot (2|k|)^{-1} ) & \text{for } k < 0.
\end{cases}
\end{split}
\end{equation}
Define $G$ as an extension of $1_C$ so that $u_{i,k} \stackrel{G}{\longmapsto} u_{i,k+1}$ and  $v_{i,k} \stackrel{G}{\longmapsto} v_{i,k+1}$. Thus, above each
endpoint $a_i$ the $u_{i,k}$'s runs up the $(a_i,\ell_i/m)$'s with $m$ even, jumps from $(a_i,\ell_i/2)$ to $(b_i,\ell_i)$,
and then moves down the  $(b_i,\ell_i/m)$'s with $m$ odd. The $v_{i,k}$'s provide a similar path
from $b_i$ to $a_i$. Since for any $\ep > 0$ at most finitely many points move
a distance more than $\ep$, it follows that $G$ and its inverse are continuous.

It is clear that for any $i$ the points of $\{ (a_i,\ell_i/m), (b_i,\ell_i/m) : m \in \N \}$ all lie in a single chain
component. Given $\ep > 0$, it is clear that we can get from a point $x \in C$ to a point $y \in C$ by an $\ep$-chain
jumping across the gaps of length less than $\ep$ which occur between $x$ and $ y$. For the finite number of
remaining gaps we use the isolated point orbits to get across. Hence, $G$  is chain transitive on $C^i$.

For an arbitrary space $X$, we again use a continuous surjection $h : C \to X$ from Proposition \ref{prop2.06}.
Let $X^i$ be the quotient space $C^i/h$ obtained by attaching $X$ via $h$. Then, $G \cup 1_X$ on $C^i \cup X$ factors through
the quotient map $q_h$ to define a homeomorphism $G^i$. Because $q_h$ maps $G$ on $C^i$ onto $G^i$ on $X^i$ it follows
that $G^i$ is chain transitive by Proposition \ref{prop2.02b} (d). Because $q_h : C^i \setminus C \to X^i \setminus X$ is
a homeomorphism, it follows that $X^i$ is an isolated point extension of $X$.

Because $(X^{(K)}, X)$ is a $K$-extension pair and the point extension is essentially unique, there exists
$\pi : (X^{(K)}, X) \to (X^i, X)$, a map of pairs rel $X$, such that the fiber $\pi^{-1}(x)$ is homeomorphic to $K$ for
every $x \in X^i \setminus X$. For each such $x$ let $G^K$ restrict to a homeomorphism from $\pi^{-1}(x)$ to
$\pi^{-1}(G^i(x))$. By Lemma \ref{lem3.05} $G^K$ and its inverse are continuous. By Lemma \ref{lem3.10} $G^K$ is
chain transitive because $G^i$ is.
\end{proof}

\begin{remark}
By Proposition \ref{prop2.03} (b) there is no topologically transitive homeomorphism on a space with
infinitely many isolated points. Hence, the result in (a) above is the best we can hope for. In particular, we see that any chain
transitive homeomorphism on $X$ can be extended to a system in which $X$ is an omega limit set.
\end{remark}

Recall from Proposition \ref{prop2.05} (e) that if $X$ is connected, then $1_X$ is chain transitive.

Now we can prove a slight extension of Theorem \ref{theo1.01}.

\begin{cor}\label{cor3.12}  For a space $X$ let $X_1$ be the closure of the union of
all components which meet $\ol{\Iso(X)}$. If $X_1$ is a proper, clopen, nonempty subset of $X$, then $X$ is $H(X)$-decomposable
and so $X$ admits no chain transitive homeomorphism. If $X_1$ is  not a proper clopen subset of $X$
and the open set $X \setminus X_1$ is empty or zero-dimensional, then $X$ admits a chain transitive
 homeomorphism. \end{cor}

\begin{proof}
The sets $\ol{\Iso(X)}$ and $X_1$ are $h_X$-invariant, so if $X_1$ is proper, clopen and nonempty,
 then $X$ is $h_X$-decomposable.

Assume that $X_1$ is not a proper clopen subset of $X$. If $\Iso(X)$ is finite and nonempty, then $X = X_1 = \Iso(X)$
and we can define
$f$ so that $X$ consists of a single periodic orbit. If $\Iso(X) = \emptyset$, then $X_1 = \emptyset$ and
$X = X \setminus X_1$ is
zero-dimensional and perfect, and so $X$ is a Cantor set. Hence, $X$ admits a topologically transitive homeomorphism.

Now assume that $\Iso(X)$ is infinite, so $A = \ol{\Iso(X)} \setminus \Iso(X)$ is nonempty. Clearly,
$(\ol{\Iso(X)}, A)$ is a  point extension pair, and by Theorem \ref{theo3.11}(c)
there exists a chain transitive homeomorphism $f_1$ on $\ol{\Iso(X)}$
which is the identity on $A$. Extend $f_1$ to be the identity on $X_1 \setminus \Iso(X)$. Thus, $f_1$ is the identity
on every nontrivial component of $X$ which meets $\ol{\Iso(X)}$. It follows from Proposition \ref{prop2.05} (e)
that all of these are contained in the
chain component of $f_1$ which contains all of $\ol{\Iso(X)}$. As this chain component is closed, it must contain all of $X_1$.
That is, $f_1$ on $X_1$ is chain transitive.
If $X = X_1$, then we are done.

Otherwise, the nonempty, open, zero-dimensional set $X \setminus X_1$ is not closed and contains no isolated points.
So $X_2 = \ol{X \setminus X_1}$ is perfect and $B = X_2 \cap X_1$ is nonempty subset of $X_1$ disjoint from
$\Iso(X)$. We see that $(X_2, B)$ is a Cantor set extension pair, so by Theorem \ref{theo3.11}(b) there exists a topologically transitive
homeomorphism $f_2$ on $X_2$ which restricts to the identity on $B$.

The concatenation $f = f_1 \cup f_2$ is a homeomorphism on $X$. Since $f_1$ and $f_2$ are chain transitive and
$X_1 \cap X_2 \not= \emptyset$ it follows that all of $X$ is contained in a single chain component, i.e.
$f$ is chain transitive.
\end{proof}

To extend these results we need some simple lifting facts.

\begin{lem}\label{lem3.13} Let $f_i \in H(X_i)$ for $i = 1,2$ and let $\pi : X_1 \to X_2$ be a
continuous surjection mapping $f_1$ to $f_2$. Assume that $(\pi \times \pi)^{-1}(1_{X_2}) \subset \CC f_1$.
That is, each fiber of $\pi$ is entirely contained in a single chain component of $f_1$.
\begin{itemize}
\item[(a)] Both $f_1$ and $f_2$ are chain recurrent, i.e. $$1_{X_1} \subset \CC f_1 \quad  \mathrm{and} \quad 1_{X_2} \subset \CC f_2.$$
\item[(b)] The space $X_1$ is $f_1$-decomposable iff $X_2$ is $f_2$-decomposable.
\item[(c)] The chain relations satisfy $$\CC f_2 = (\pi \times \pi)(\CC f_1) \quad  \mathrm{and} \quad \CC f_1 = (\pi \times \pi)^{-1}(\CC f_2).$$
\item[(d)] The homeomorphism $f_1$ is chain transitive iff $f_2$ is chain transitive.
\end{itemize}
\end{lem}

\begin{proof}
(a) Let $E_{\pi} = (\pi \times \pi)^{-1}(1_{X_2})$.  It is a closed equivalence relation and so
contains $1_{X_1}$.  Hence, $1_{X_1} \subset E_{\pi} \subset \CC f_1$.

Because $\pi$ is surjective, $1_{X_2} = (\pi \times \pi)(1_{X_1})$.  Since $\pi$ maps $f_1$ to $f_2$, it maps
$\CC f_1$ to $\CC f_2$ by Proposition \ref{prop2.02b}(d).
Hence,
$$ 1_{X_2} \ = \ (\pi \times \pi)(1_{X_1}) \ \subset \ (\pi \times \pi)(\CC f_1) \ \subset \ \CC f_2.$$

(b) If $B$ and its complement are proper, clopen, forward $f_2$-invariant subsets of $X_2$ then because $\pi$ is
surjective, $\pi^{-1}(B)$ and its complement are proper, clopen, forward $f_1$-invariant subsets of $X_1$.

Now assume that $A$ and its complement are proper, clopen, forward $f_1$-invariant subsets of $X_1$. Since each is
forward $\CC f_1$-invariant, it follows that each is saturated by the equivalence relation $E_{\pi} \subset \CC f_1$.
Hence, $\pi(A)$ and $\pi(X_1 \setminus A)$ are disjoint closed sets with union $\pi(X_1) = X_2$. That is, they are
 complementary clopen sets. Furthermore, they are forward $f_2$-invariant.

 (c)  As mentioned above,   $(\pi \times \pi)(\CC f_1) \subset \CC f_2$. Now assume $(x_1,x_2) \not \in \CC f_1$.
 Since $X_1 = |\CC f_1|$ by (a), Proposition \ref{prop2.01}(c) implies there is an attractor $A$ for $f_1$ which
 contains $x_1$ but not $x_2$, and by Proposition \ref{prop2.01}(d) $A$ is a clopen  $\CC f_1$-invariant set.
 Hence, it is saturated by $E_{\pi}$. So $\pi(A)$ is clopen and $f_2$-invariant, and therefore also $\CC f_2$-invariant.
 Now, $\pi(x_1) \in \pi(A)$, and $x_2 \not\in A = \pi^{-1}(\pi(A))$ implies $\pi(x_2) \not\in \pi(A)$.
It follows that
 $(\pi(x_1),\pi(x_2)) \not \in \CC f_2$. Thus, the complement of $\CC f_1$ in $X_1 \times X_1$ maps into the complement
 of $\CC f_2$. Since $\pi \times \pi$ is surjective, the equations of (c) follow.

 (d) Immediate from (c) and the surjectivity of $\pi$.
\end{proof}

 For any space $X$, the chain relation $\CC 1_X$ is a closed equivalence relation with equivalence classes the components of $X$.
 Let $[X]$ be the zero-dimensional space of components, the quotient space for this equivalence relation with quotient map
 $\pi_X : X \to [X]$, see Lemma \ref{lem2.04}(b) and Proposition \ref{prop2.01}(d).
  There is a natural homomorphism $[\cdot] : H(X) \to H([X])$ with $\pi_X$ mapping $f$ to
 $[f]$ for $f \in H(X)$. Thus, $H(X)$ acts on $[X]$ and we let $[h_X] = \bigcup \{ [h] : h \in H(X) \}$
 be the associated relation on $[X]$.

\begin{prop}\label{prop3.14}
\leavevmode
\begin{itemize}
\item[(a)] A space $X$ is  $h_X$-decomposable iff $[X]$ is $[h_X]$-decomposable.
\item[(b)] If $f$ is a chain recurrent homeomorphism on a space $X$, then
$(\pi_X \times \pi_X)^{-1}(1_{[X]}) \ \subset \ \CC f$. In that case, the following are equivalent:
\begin{itemize}
\item[(i)] The map $f$ is chain transitive.
\item[(ii)] The space $X$ is  $f$-indecomposable.
\item[(iii)] The map $[f]$ is chain transitive.
\item[(iv)] The space $[X]$ is  $[f]$-indecomposable.
\end{itemize}
\end{itemize}
\end{prop}

\begin{proof}
(a) This is clear because any clopen set is saturated by the equivalence relation $\CC 1_X$.

(b)  Since the space of chain components is zero-dimensional, every
connected set of chain recurrent points is contained in a single chain component. If $f$ is chain recurrent, then
every point is chain recurrent, so each component is contained in a single chain component. It follows that
$(\pi_X \times \pi_X)^{-1}(1_{[X]}) \ \subset \ \CC f$. Since $f$ and $[f]$ are chain recurrent, the equivalences
(i) $\Leftrightarrow$ (ii) and (iii) $\Leftrightarrow$ (iv) follow from Proposition \ref{prop2.05} (d). The implication
(i) $\Rightarrow$ (iii) holds because $\pi$ maps $f$ to $[f]$. The converse (iii) $\Rightarrow$ (i) follows from Theorem \ref{lem3.13}.
\end{proof}

Thus, a chain transitive homeomorphism on $[X]$ lifts to a chain transitive homeomorphism iff it lifts to a chain recurrent
homeomorphism.

Recall that a component $K$ of $X$ is an \emph{isolated component} if it is a clopen subset of $X$.

For a space $X$ let $\II_X$ denote the set of isolated components. Two isolated components $K_1$ and $K_2$ are \emph{$H(X)$-equivalent} if
they are homeomorphic or, equivalently, if there exists $g \in H(X)$ such that $g(K_1) = K_2$. Let $I_X$ be the
set of $H(X)$-equivalence classes in $\II_X$.
For $i \in I_X$ let $Q_i$ be the union of the isolated components in  $i$, and let $Q$ be the
union of all of the isolated components, so that
$Q$ is the disjoint union of the $Q_i$'s. Thus, $Q$ and each of the $Q_i$'s are open subsets of $X$.

\begin{lem}\label{lem3.15} If $A$ is a clopen $H(X)$-invariant set which meets some $\ol{Q_i}$, then it
contains $\ol{Q_i}$. In particular, if all the isolated components are homeomorphic to
one another and the union of the isolated components is dense, then $X$ is $H(X)$-indecomposable. \end{lem}

\begin{proof}
Since $A$ is open, it meets some isolated component $K \in i$. Since it is clopen, it contains $K$.
If $K_1 \in Q_i$, then there exists $g \in H(X)$ such that $g(K) = K_1$, and so $H(X)$-invariance implies $K_1 \subset A$.
Since $Q_i \subset A$ and $A$ is closed, we get $\ol{Q_i} \subset A$.
\end{proof}

We will say that $X$ satisfies the \emph{diameter condition on isolated components} if for every $\ep > 0$
there are only finitely many isolated components with diameter greater than $\ep$.

The following is the furthest we can extend Corollary \ref{cor3.12}.

\begin{theo}\label{theo3.16} For a space $X$, let $X_1$ be the closure of the union of
all components which meet the closure of the union of all isolated components.
Assume that $X$ satisfies the diameter condition on isolated components
and that the open set $X \setminus X_1$ is empty or zero-dimensional.
\begin{itemize}
\item[(a)] Either $X$ is $H(X)$-decomposable or $X$ admits a chain transitive homeomorphism.
\item[(b)] If  $X_1$ is a proper clopen subset of $X$, then $X$ is $H(X)$\hyp{}decomposable,
and so $X$ admits no chain transitive homeomorphism.
\item[(c)] If $X_1$ is  not a proper clopen subset of $X$ and all the isolated components are homeomorphic to one another,
 then $X$ admits a chain transitive  homeomorphism.
\end{itemize}
\end{theo}

\begin{proof}
(b) Obvious since $X_1$ is $H(X)$-invariant.

In (a) and (c) the extension from the closure of the isolated components to the rest of $X$ proceeds just as in Corollary
\ref{cor3.12}.  So from now on we will assume that the union $Q$ of the isolated components is dense.

(c) If there are only $n$ isolated components, then their  union $Q$ is clopen and so is all of $X$.
By assumption they are all homeomorphic to some common space $K$.
We can choose a homeomorphism with $f^n = 1_X$, and so that each periodic orbit meets each component. Clearly, $f$ is
chain transitive.

Now we may assume that $\II_X$ is infinite, and let $A = X \setminus Q$.  Thus, $A$ is a nonempty, closed
nowhere dense set.  By Theorem \ref{theo3.08} (b) $(X, A)$ is a $K$-extension pair. By Theorem \ref{theo3.11} (c) $X$ admits a chain transitive homeomorphism rel $A$.

(a) We may assume that there is more than one equivalence class in $I_X$, for otherwise we are in case (c).
If any $i \in I_X$ is finite, then $Q_i$ is a clopen $H(X)$-invariant set, so $X$ is $H(X)$-decomposable
since $X \not= Q_i$ by the assumption that $I_X$ contains more than one class.

Now assume that every equivalence class in $I_X$ is infinite.  Then the closure of each open $H(X)$-invariant set $Q_i$ meets
$A$, and we let $A_i = \ol{Q_i} \cap A$.  If $i \not= j$, then $\ol{Q_i} \cap \ol{Q_j} \subset A$, and so this intersection equals
$A_i \cap A_j$.

Applying the argument for (c) to $\ol{Q_i}$, there exists a homeomorphism $f_i$ on $\ol{Q_i}$ which is chain transitive
and which restricts to the identity on $A_i$.

Let $f$ on $X$ equal $f_i$ on $\ol{Q_i}$ and the identity on $A$. Because the diameter condition holds, we can apply Lemma
\ref{lem3.10} to see  that
$f$ is a homeomorphism on $X$.  Since each $f_i$ is chain transitive, $\ol{\bigcup_i \ \{ Q_i \times Q_i \} } \subset \CC f$.

Let $x \in X$ and $g \in H(X)$.  Since $Q = \bigcup_i \ Q_i$ is dense in $X$, there is a sequence $\{ x_k \in Q_{i_k} \}$
which converges to $x$. Then $g(x_k) \in Q_{i_k}$, and so $(x,g(x)) \in \ol{\bigcup_i \ Q_i \times Q_i}$.
Hence, $h_X \subset \CC f$. In particular, $1_X \subset h_X$ implies that $f$ is chain recurrent. So by
Proposition \ref{prop2.05} (d) $f$ is chain transitive iff $X$ is $f$-indecomposable. Since $h_X \subset \CC f$, an
$f$-decomposition, which is a $\CC f$-decomposition, is also an $h_X$-decomposition. Hence, if $X$ is $f$-indecomposable, then
it is $h_X$-indecomposable, i.e. $X$ is $H(X)$-indecomposable. On the other hand, if $X$ is $H(X)$-decomposable, then
there does not exist any chain transitive homeomorphism on $X$.
\end{proof}

As we will see below, the diameter condition on isolated components is essential for this result.

\begin{ex}\label{ex3.17}
We construct $X$ so that
\begin{itemize}
\item The connected components of $X$ are all homeomorphic to $[0,1]$.
\item The space $X$ is $H(X)$-decomposable.
\item There are no isolated components.
\end{itemize}

Let $C$ be a Cantor set and $S = \{ a_n : n \in \N \}$ a sequence of distinct points in $C$ with closure $A$ in $C$.
Define
\begin{equation}\label{3.09a}
\begin{gathered}
I_n  \ = \  \{ (a_n,t) : 0 \leq t \leq n^{-1} \} \ \subset \ C \times [0,1] \quad \text{for } n \in \N, \\
C_0  \ = \  C \times \{ 0 \}, \qquad C_+ \ = \  C_0 \ \cup \  \bigcup_n \ I_n, \\
C_{\pm} \ = \  C_+ \ \cup \ C \times [-1,0],\\
A_0  \ = \  A \times \{ 0 \}, \qquad A_+  \ = \  \ol{ \bigcup_n \ I_n},  \\
A_{\pm} \ = \ A_+ \ \cup \ A \times [-1,0].
\end{gathered}
\end{equation}
Each $I_n^{\circ} = I_n \setminus C \times \{ 0 \}$ is open in $C_{\pm}$, and so the points of each $I_n^{\circ}$
have connected neighborhoods. Hence, $A,A_+$ and $A_{\pm}$ are $H(C_{\pm})$-invariant. Thus, if $A$ is a proper clopen
subset of $C$, then $C_{\pm}$ is $H(C_{\pm})$\hyp{}decomposable.  Observe that every component is homeomorphic to the
unit interval and the first coordinate projection maps $C_{\pm}$ onto the Cantor set $C$.

Also, if $A$ is a proper clopen set, then $C \times [-1,0]$ admits chain transitive homeomorphisms, but the factor
$X = C \cup A \times [-1,0]$ is $H(X)$-decomposable.

A homeomorphism $f$ on $C$ can be extended to a homeomorphism $F_+$ of $C_+$ iff $A$ is $f$-invariant. In that case, we can then define
$F_+$ by using any orientation-preserving homeomorphism from $I_x$ to $I_{f(x)}$, i.e. one which maps
$(x,0)$ to $(f(x),0)$. Here $I_x = I_n$ for $x = a_n$ and $= \{ (x,0) \}$ if $x \not\in S$. Continuity at points
of $A_+^{\circ}$ is clear, and if $x \in C$, then for every $\ep$ there is a neighborhood $U$ of $x$ in $C$
so that $y \in U \setminus \{ x \}$ implies that the length of the interval $I_{f(y)}$ is less than $\ep$. This implies
continuity at $(x,0)$. Notice that if $\OO(x)$ is infinite, then
\begin{equation}\label{3.10a}
\lim_{|n| \to \infty} \ |I_{f^n(x)}| \ = \ 0,
\end{equation}
where $|J|$ denotes the length of an interval $J$. This says that any pair $(x,t_1),(x,t_2) \in I_x$ is
\emph{asymptotic} for $F_+$ and $(F_+)^{-1}$  with
\begin{equation}\label{3.11a}
\begin{split}
\om F_+(x,t_1) \ &= \ \om F_+(x,t_2) \ = \ \om f(x) \times \{ 0 \} \ \subset \ C_0, \\
\al F_+(x,t_1) \ &= \ \al F_+(x,t_2) \ = \ \al f(x) \times \{ 0 \} \ \subset \ C_0.
\end{split}
\end{equation}

Now assume that $A$ is not clopen. If $f$ is chain transitive and
every point of $A$ has an infinite orbit, then any extension $F_+$ to
$C_+$ is chain transitive. Observe that $F_+$ is chain transitive if, whenever $a_n$ is a periodic point for $f$, we
define $F_+$ by using the unique linear, orientation-preserving homeomorphism from $I_{a_n}$ to $I_{f(a_n)}$.
On the other hand, if $f(a_1) = a_1$ and on $I_{a_1}$ we define $F_+$ by $(a_1,t) \mapsto (a_1,t^2)$ then
$(a_1,1)$ is a repelling fixed point for $F_+$ on $I_{a_1}$ and hence for $F_+$ on $C_+$. If $F_+$ is chain transitive, then
we can
obtain a chain transitive extension $F_{\pm}$ on $C_{\pm}$ by using $f \times 1_{[-1,0]}$ on $C \times [-1,0]$.
\end{ex}

\begin{ex}\label{ex3.18} We construct $X$ so that
\begin{itemize}
\item The connected components of $X$ are all homeomorphic to $[0,1]$.
\item The space of connected components, $[X]$, consists of a convergent sequence and its limit, and the union of isolated components in $X$ is dense.
\item It is $f$-decomposable for any $f \in H(X)$ but not $H(X)$\hyp{}decomposable.
\end{itemize}

The space $X$ we construct is $H(X)$-indecomposable by the first two properties and by Lemma \ref{lem3.15}.

For every $n \in \N$ we define $I_n = [0,n^{-1}]$ and a continuous function $t_n \colon I_n \to I = [0,1]$ so that, for integers $i = 0, \dots, (2n)!$,
\[
t_n \left( \frac{i}{n(2n)!} \right) \ = \
\begin{cases}
0 & \text{when $i$ is even}, \\
1 & \text{when $i$ is odd},
\end{cases}
\]
and the rest of the values are defined by linear interpolations.

Each interval $I_n$ contains $(2 n)!$ intervals $\{ I^i_n : i = 1, \dots, (2 n)! \}$ of equal length, each of which is mapped by $t_n$ onto $I$. We can further subdivide each $I^i_n$ into intervals $\{ I^{i,j}_n : j = 1, \dots, n \}$ of equal length so that each is mapped to a subinterval of $I$ of length $n^{-1}$ by $t_n$. The corresponding restrictions of $t_n$ are denoted by $t_n^i = t_n|_{I_n^i}$ and $t_n^{i,j} = t_n|_{I_n^{i,j}}$.

Regarding the functions $t_n$, $t_n^i$ and $t_n ^{i,j}$ as closed subsets of $I_n \times I$, we define
\begin{equation}\label{3.12a}
\begin{split}
X_n  \ &= \ \{ n^{-1} \} \times t_n, \quad n \in \N, \\
X_\infty \ &= \ \{(0,0) \} \times [0,1], \\
X \ &= \ \bigcup_n X_n \cup X_\infty.
\end{split}
\end{equation}
The space $X$ is clearly a closed,
bounded subset of $\R^3$, and the space $[X]$ can be identified
with $\pi(X)$, where $\pi : X \to \{ n^{-1} : n \in \N \} \ \cup \ \{ 0 \} $ is the projection to the first coordinate. The union of the isolated components is clearly dense in $X$. In addition, for all appropriate $n$, $i$ and $j$, we define
\begin{equation}\label{3.12b}
X_n^i \ = \ \{ n^{-1} \} \times t_n^i \quad \text{ and } \quad X_n^{i,j} \ = \ \{ n^{-1} \} \times t_n^{i,j}.
\end{equation}
Note that each $X_n$ is a union of the line segments $X_n^i$. Similarly, each $X_n^i$ is a union of the line segments $X_n^{i,j}$. The diameter of each $X_n^i$ is greater than $1$, and the diameter of each $X_n^{i,j}$ is less than $2n^{-1}$. The arc length of $X_n$ is greater than $(2n)!$.

Suppose that $h : X_n \to X_m$ is a homeomorphism
for some $ n < m < \infty$. Then for some $i = 1, \dots, (2n)!$, $j = 1, \dots, n$ and $k = 1, \dots, (2m)!$, it must happen that
$h(X_n^{i,j}) \supset X_m^k$. If not, then each $h(X_n^{i,j})$ meets at most two of the segments $X_m^i$, so the arc length of each mapped segment $h(X_n^{i,j})$ is less than $4$. The arc length of $h(X_n)$, which is the sum of the arc lengths of the $(2n)! n$ mapped segments $h(X_n^{i,j})$, is less than $4 (2n)! n$.
Since $4 (2n)! n < (2m)!$, the map $h$ could not be surjective. It follows that there exists a pair of points
$u, v \in X_n$ with $d(u,v) \leq 2 n^{-1}$ but with $d(h(u),h(v)) \geq 1$.

Now if $f \in H(X)$, then it follows that the induced homeomorphism $[f]$ on the space $[X]$ is the identity on all
but finitely many points. For if not, then by replacing $f$ by $f^{-1}$ if necessary, we can assume that
there are sequence $(m_i)$ and $(n_i)$ in $\N$ tending to infinity such that $f(X_{n_i}) = X_{m_i}$ and $m_i > n_i$ for all $i$. Hence, there
exist $u_i, v_i \in X_{n_i}$ with $d(u_i,v_i) \leq 2 n_i^{-1}$ but with $d(f(u_i),f(v_i)) \geq 1$. Thus, there are
convergent subsequences of $\{ u_i \}$ and
$\{ v_i \}$ with a common limit in $X_{\infty}$.  Hence, $f$ could not extend to a continuous function on
all of $X$.

Thus, there exists $N \in \N$ such that $f(X_n) = X_n$ for all $N \leq n \leq \infty$. Since the isolated components $X_n$ are
invariant for $n$ large enough, $X$ is $f$-decomposable.
\end{ex}

\begin{ex}\label{ex3.19} We construct spaces $X$ and $X^+$ so that
\begin{itemize}
\item The isolated components of $X$ and $X^+$ are all homeomorphic to one another and their union is dense.
\item $X^+$ is  $H(X^+)$-indecomposable but it is $f$-decomposable for all $f \in H(X^+)$.
\item There exists $f \in H(X)$ such that $X$ is $f$-indecomposable, but no  $f \in H(X)$ is chain transitive.
\end{itemize}

 Let $Z = \{ x_n : n \in \Z \} \subset I = [0,1]$ with
\[
x_n =
\begin{cases}
1 - (n+2)^{-1} & \text{for } n = 0, 1, \dots, \\
(|n| + 2)^{-1} & \text{for } n = -1, -2, \dots.
\end{cases}
\]
Define for $n \in \Z$
\begin{equation}\label{3.13a}
\begin{gathered}
a_n \ = \ (x_n,0), \qquad b_n \ = \ (x_n, (|n|+1)^{-1}), \\
I_n \ = \ \{ (x_n,t) : 0 \leq t \leq (|n|+1)^{-1} \}.
\end{gathered}
\end{equation}
Define
\begin{equation}\label{3.14a}
\begin{gathered}
J = I \times \{ 0 \}, \qquad H = [-1,0] \times \{ 0 \} \\
C \ = \ H \ \cup \ J \ \cup \ \bigcup \{ I_n: n \in \Z \}, \\
Z_0 \ = \ \{ a_n  : n \in \Z \}, \qquad Z_1 \ = \ \{ b_n  : n \in \Z \}, \\
e_0 \ = \ (0,0), \qquad e_1 \ = \ (1,0), \qquad e_{-1} \ = \ (-1,0).
\end{gathered}
\end{equation}
We will call $C$ a \emph{comb} with handle $H$.

The group $H(C)$ fixes $e_{-1}, e_{0}$ and $e_1$. Each of the sets $Z_0$ and $Z_1$ is a single $H(C)$-orbit. The
closed sets $J$ and $H$ are $H(X)$-invariant.

If $h \in H(C)$, then
$$h(b_n) \ = \ b_{n+k} \ \Leftrightarrow \ h(a_n) \ = \ a_{n+k} \ \Leftrightarrow \ h(I_n) \ = \ I_{n+k}.$$
In that case $h(a_{n \pm 1}) = a_{n \pm 1 + k}$, and so $h$ induces a translation by $k$ on the sequences $Z_0$ and $Z_1$.
If $k  > 0$, then $e_1$ is an attractor with complementary repellor $H$, and
the reverse is true if $k < 0$. If $k = 0$, then $h$ fixes each point of $Z_0$ and of $Z_1$.

On $C$ define $T$ by
\begin{equation}\label{3.15a}
\begin{aligned}
T(e_0) \ &= \ e_0, & T(e_{\pm 1}) \ &= \ e_{ \pm 1},\\
T(a_n) \ &= \ a_{n+1}, & T(b_n) \ &= \ b_{n+1},
\end{aligned}
\end{equation}
and with $T : [a_{n-1},a_{n}] \to [a_{n},a_{n+1}]$, $T : I_n \to I_{n+1}$ and $T : H \to H$ linear
for all $n \in \Z$. Thus, $T \in H(C)$ induces
a translation by $1$ on $Z_0$ and $Z_1$.

For $n \in \Z$ let
\begin{equation}\label{3.16a}
\begin{split}
I_n^* &= I_n \cup \left \{ (x,(|n|+1)^{-1}) : \frac{x_{n-1} + x_n}{2} \leq x \leq \frac{x_{n+1} + x_n}{2} \right \}, \\
C_n &= C \cup I_n^*,
\end{split}
\end{equation}
a comb with a queer tooth $I^*_n$ at $n$ replacing $I_n$.

Observe that $C$ and each $C_n$ are connected.

If $f : C_n \to C_m$ is a homeomorphism, then
$f(I_n^*) = I_m^*$.  If $h \in H(C)$, then
$h$ extends to a homeomorphism from $C_n$ to $C_m$ iff $h(I_n) = I_m$. Thus, $C_n$ and $C_m$ are homeomorphic
for all $n,m \in \Z$.

In $\R^2 \times I$ we define
\begin{equation}\label{3.17a}
\begin{split}
X \ &= \ C \times \{ 0, 1 \} \ \cup \ \bigcup  \{ C_n \times \{ x_n \} : n \in \Z \}, \\
X^+ \ &= \ C \times \{ 1 \} \ \cup \ \bigcup  \{ C_n \times \{ x_n \} : n \in \N \},
\end{split}
\end{equation}

Thus, $X$ and $X^+$ both have a dense union of isolated components, and each isolated component is homeomorphic to $C_0$.  It follows from
Lemma \ref{lem3.15}
that $X$ is $H(X)$-indecomposable and $X^+$ is $H(X^+)$\hyp{}indecomposable.

Any homeomorphism in $H(X^+)$  restricts to a homeomorphism of
$C\times \{  1 \}$  and so
restricts to translation by $k$ on $Z_0 \times \{ 1 \}$. This means that
$h$ must map $C_n \times \{ x_n \}$ to $C_{n+k} \times \{ x_{n+k} \}$ when $n \geq N$ for $N$ sufficiently large, and we may
suppose $N + k > 0$.  This implies that $h$ maps the set of $N$ complementary components
$\{ C_n \times \{ x_n \} : 0 \leq n < N \}$ to the set of $N+k$ components $\{ C_n \times \{ x_n \} : 0 \leq n < N +k \}$.
This
requires that $k = 0$. Hence, each isolated component  $C_n \times \{ x_n \} $ with $n \geq N$ is invariant, and so $X^+$
 is $h$-decomposable.

 Thus, $X^+$ is an example of a space which is $H(X^+)$-indecomposable but such that $X^+$ is $h$-decomposable for
 every $h \in H(X^+)$.

In the case of $X$, we
 start by assuming that $h$ fixes each of the two non-isolated components, i.e. $h$ maps
 $C \times \{ \ep \}$ to itself for $\ep = 0, 1$. As before, since $h$ translates
 $Z_0 \times \{ \ep \}$ by some $k_{\ep}$ for $\ep = 0, 1$,  there exists $N \in \N$
large enough that $h$ maps $C_n \times \{ x_n \}$ to $C_{n+k_1} \times \{ x_{n+k_1} \}$ for $n \geq N$ and
$C_n \times \{ x_n \}$ to $C_{n+k_0} \times \{ x_{n+k_0} \}$ for $-n \geq N$. Again we can choose $N$ large enough that
$N + k_{\ep} > 0$ for $\ep = 0, 1$.
This implies that
$h$ maps the set of complementary components $C_{n} \times \{ x_n \}$ for $-N < n < N$ to the set of components
 $C_{n} \times \{ x_n \}$ for $-N + k_0  < n < N + k_1$. This requires that $k_0 = k_1$, and we let $k = k_0 = k_1$.

If $k > 0$, then
$C \times \{ 1 \}$ is an attractor and $C \times \{ 0 \}$ is a repellor with the reverse if $k < 0$. Hence, $h$ is
not chain transitive.  If $k = 0$, then each isolated component $C_n \times \{ x_n \}$ is invariant for $|n| \geq N$,
and so $X$ is $h$-decomposable.

The remaining possibility is that $h$ interchanges the two limit components
$ C \times \{ \ep \}$  for $\ep = 0, 1$ and then each is invariant for $h^2$.
Applying the previous argument to $h^2$, we see that for a large $N$ the components are translated by $h^2$ with a common $k$.
If $k > 0$, then $C \times \{ 1 \}$ is an attractor for $h^2$ while $C \times \{ 0 \}$ is a repellor. But
since $h$ commutes with $h^2$ this would imply that $C \times \{ 0 \} = h(C \times \{ 1 \})$ would be an
attractor for $h^2$ as well, which it is not.  Similarly, $k < 0$ leads to a contradiction.
It follows that in this interchange case $k = 0$, and so $C_n \times \{ x_n \}$ is $h^2$-invariant
for $|n|$ sufficiently large.
Hence, for each such $n$ the set $C_{n} \times \{ x_{n} \} \cup h(C_n \times \{ x_n \})$ is clopen and $h$-invariant, so
$X$ is $h$-decomposable when $h$ interchanges the ends.

Finally, extend $T : C \to C$ by $T_n : C_n \times \{ x_n \} \to C_{n+1} \times \{ x_{n+1} \} $ for all $n \in \Z$ and
by $T \times 1_{\{ 0, 1 \} }$ on $C \times \{ 0 , 1 \}$ to obtain a homeomorphism (with $k = 1$) with respect to which $X$ is not
decomposable.

Thus, $X$ is an example of a space with an element $h \in H(X)$ so that $X$ is not $h$-decomposable, but nonetheless
there is no chain transitive homeomorphism in $H(X)$.
\end{ex}

\section{Spaces with All Homeomorphisms Chain Transitive}\label{allct}

Having considered spaces which admit no chain transitive homeomorphisms, we turn to the opposite extreme to consider
spaces such that every homeomorphism is chain transitive. By Proposition \ref{prop2.05} (e) the identity $1_X$ is
chain transitive iff $X$ is connected, and if $f \in H(X)$ with $X$ connected,
then $f$ is chain transitive iff it is
chain recurrent.

A space $X$ is called \emph{rigid} if $1_X$ is the only homeomorphism on $X$, i.e. the group $H(X)$ is trivial.
Such spaces were introduced and constructed by de Groot and  Willie \cite{dG-W}.  For a connected rigid space the only
element of $H(X)$ is chain transitive.  We construct some more interesting examples by using rigid spaces as tools.
We need a pairwise disjoint sequence $\{ Z_n \}$ of connected, locally connected spaces such that
\begin{itemize}
\item[($\ast$)] For any nonempty open subset $U$ of $Z_n$ and any disk $I^k$,
there does not exists a homeomorphism of $U \times I^k$ onto any subset of $(Z_n \setminus U) \times I^k$ or onto
any subset of $Z_m \times I^k$ with $m \not= n$.
\end{itemize}
Furthermore, for every $n \in \N$  and for any finite $F \subset Z_n$, the set $Z_n \setminus F$ has only finitely many components
and for any positive integer $N$ there is a subset $F \subset Z_n$ of cardinality $N$ such that
$Z_n \setminus F$ is connected.

Condition ($\ast$) is a slight strengthening of the condition on a space called \emph{strongly chaotic} in \cite{C-C}.
We construct such a sequence in the Appendix, Section \ref{const}. For each $Z_n$ we choose a pair of
distinct points $e_n^-, e_n^+ \in Z_n$.
Our examples are obtained by using such rigid spaces instead of the unit interval in some common constructions.

\begin{ex}\label{ex4.01} We can think of the real line as a graph with $\Z$ as the set of vertices and intervals $[n,n+1]$ as edges.
Now let $Z$ be one of the chaotic spaces described above with points $e^-, e^+ \in Z$. Let $X_0$ be the
quotient space of $\Z \times Z$ with $(n,e^+)$ identified with $(n+1,e^-)$. If $t$ is the translation homeomorphism on
$\Z$ with $t(n) = n+1$, then $t$ has a unique extension $t$ to $X_0$ which is the quotient of $t \times 1_Z$.
The only homeomorphisms on $X_0$ are the iterates $t^n$.  Let $X$ be the one-point compactification of $X_0$ with
the additional point $\infty$. Let $t \in H(X)$ be the unique homeomorphism extension of $t$ on $X_0$ and so of $t$ on $\Z$.
Since $X$ is connected, $1_X$ is chain transitive.  For any $n \not= 0$, we have $\{ \infty \} \ = \ \om (t^n)(x) \ = \ \al (t^n)(x) $
for all $x \in X$, and so $t^n$ is chain transitive.  In this case, $H(X)$ is isomorphic to $\Z$.

In general, suppose that $G$ is a finitely generated (and hence countable) group. Following de Groot \cite{dG}   we
let $X_0$ be the Cayley graph with rigid spaces as linking edges.  That is, let $\{g_1,..,g_n \}$ be a list of generators
for $G$, and let $\{ Z_1, \dots, Z_n \} $ be distinct strongly chaotic spaces as above, with chosen pairs of points.
Let $X_0$ be the quotient space of $G \cup [G \times (\bigcup_{i=1}^n \ Z_i)]$ with $(g,e_i^+)$ identified with
$(g_ig,e_i^-)$ for $g \in G, i = 1, \dots, n$ and $(g,e_i^-)$ identified with $g \in G$ for $i = 1, \dots, n.$
Because there are only finitely many generators, the space $X_0$ is
locally compact and the set of vertices $\{ g \in G \}$ is invariant with respect to any homeomorphism $h$. Furthermore,
if  $g \in G$, then $h(g_ig) = g_ih(g)$, and so $h$ commutes with all left translations. It follows that, on $G$, the mapping $h$
is the right translation $r_v$, where $v = h(u)$ and $u$ is the identity element of $G$. Thus, on $X_0$, the mapping
$h$ is the quotient of the map $r_v \cup [r_v \times 1_{\cup_i \ Z_i}]$. Let $X$ be the one-point compactification of $X_0$,
and let $r_v$ denote the extension of $h$ to $X$. If $v$ is of finite order $k$, then $(r_v)^k = 1_X$ and so $r_v$ is
chain transitive on $X$. If $v$ is of infinite order, then $\{ \infty \} \ = \ \om (r_v)(x) \ = \ \al (r_v)(x) $, and
so  $r_v$ is chain transitive on $X$ in this case as well. The group $H(X)$ is isomorphic to the discrete group $G$
by $v \mapsto r_{v}^{-1}$.
\end{ex}

In these cases, the homeomorphism group is discrete.  It is possible to obtain rather large non-discrete groups.
We first review some standard topology constructions.

A \emph{pointed space} is a pair $(X,x)$ consisting of a  space with a chosen
\emph{base point} $x \in X$. We let $H(X,x)$ denote the closed subgroup of $H(X)$ consisting of those
homeomorphisms which fix $x$. A space $Y$ can be regarded
as a pointed space with base point an isolated point not in $Y$. If $(X_1,x_1)$ and $(X_2,x_2)$ are pointed spaces and $f : X_1 \to X_2$ is a function, we use the notation $f : (X_1,x_1) \to (X_2,x_2)$ to mean that $f(x_1) = x_2$.

If $A$ is a nonempty closed subset of a space $X$, then the space $X/A$ with $A$ \emph{smashed to a point} is the quotient
space of $X$ with respect to the closed equivalence relation $1_X \ \cup \ A \times A$. Thus, the quotient map
$q : X \to X/A$ is a homeomorphism between the open sets $X \setminus A$ and $X/A \setminus \{ x_A \}$ with $x_A$ the
point which is the image of $A$.

Given two pointed spaces $(X_1,x_1), (X_2,x_2)$, their \emph{smash product} is
	\[(X_1,x_1) \# (X_2, x_2) = (X_{12},x_{12}),
\]
a pointed space consisting of the product
$X_1 \times X_2$ with the \emph{wedge} $X_1 \times \{ x_2 \} \ \cup \ \{ x_1 \} \times X_2$ smashed to the point $x_{12}$.
We can also define the smash product of a pointed space $(X_1,x_1)$ and any space $X_2$ as $(X_1,x_1) \# X_2 = (X_{12},x_{12})$, where $X_{12}$ is the product $X_1 \times X_2$ with $\{ x_1 \} \times X_2$ smashed to the point $x_{12}$.
Notice that in this case we can regard the space $X_{12}$ as the one-point compactification of $(X_1 \setminus x_1) \times X_2$.
The projections $\pi_1 : (X_1,x_1) \# X_2 \to (X_1,x_1)$ and $\pi_2 : X_1 \setminus \{ x_1 \} \times X_2 \to X_2$ are open
and surjective.

We can define the smash product of two continuous functions once we fix base points from the domains. For $i = 1,2$, let $X_i$ and $Y_i$ be spaces, $f_i : X_i \to Y_i$ a continuous function and $x_i \in X_i$ a base point. We set $y_i = f_i(x_i)$ to obtain a pointed space $(Y_i,y_i)$ for $i = 1,2$. Let $(X_{12},x_{12}) = (X_1,x_1) \# (X_2,x_2)$ and $(Y_{12},y_{12}) = (Y_1,y_1) \# (Y_2,y_2)$, and let $q : X_1 \times X_2 \to X_{12}$ and $r : Y_1 \times Y_2 \to Y_{12}$ be the quotient maps. We define a continuous function $f = (f_1,x_1) \# (f_2,x_2)$ from $(X_{12},x_{12})$ to $(Y_{12},y_{12})$ by the formula $f \circ q = r \circ (f_1 \times f_2)$. We can also define the smash product of two functions when only one of the domains has a base point. If $(V,v) = (X_1,x_1) \# X_2$ and $(W,w) = (Y_1,y_1) \# Y_2$ and if $s : X_1 \times X_2 \to V$ and $t : Y_1 \times Y_2 \to W$ are the quotient maps, then we define the continuous function $g = (f_1,x_1) \# f_2$ from $(V,v)$ to $(W,w)$ by $g \circ s = t \circ (f_1 \times f_2)$.

\begin{ex}\label{ex4.02} We construct a space such that every homeomorphism is chain transitive and the homeomorphism
group contains a nontrivial path-connected subgroup.

Let $(Z,e)$ be a chaotic space $Z$, as above,
with base point $e \in Z$ such that $Z \setminus \{ e \}$ is connected.
Let $W$ be a connected, compact manifold (perhaps with boundary) of positive dimension.
Let $(X,e_X) = (Z,e) \# W$.
If $(z,w) \in Z \times W$ and $h$ is a homeomorphism from an open set containing $(z,w)$ into $Z \times W$,
then $h(z,w) = (z_1,w_1)$ with $z, z_1 \in Z \setminus \{ e \}$ implies $z = z_1$.
If not, then we can choose disk neighborhoods of $w$ and $w_1$ each homeomorphic
to $I^k$ with $k$ the dimension of $W$, and we can choose
 disjoint  open neighborhoods $U$ of $z$ and $U_1$ of $z_1$ so that $h$ induces a homeomorphism from
$U \times I^k$ onto a subset of $U_1 \times  I^k$, and this contradicts Condition ($\ast$). If
$h(z,w) = e_X$ with $z \in Z \setminus \{ e \}$, then $h$ would map points $(z_2,w_2)$ close to $(z,w)$ to points
$(z_1,w_1)$ with $z_1$ close to $e$. This does not happen by the previous argument. Thus, it follows
that $\pi_1 \circ h \ = \ \pi_1$, where $\pi_1$ is the projection to $(Z,e)$. This implies that $H(X) = H(X,e_X)$, and $\pi_1$ maps every $h \in H(X)$ to $1_Z$.

Since the homeomorphisms of $X$ leave the preimages of $\pi_1$ invariant, it follows that every $h \in H(X)$ is of the form
$h(z,w) = (z,q(z)(w))$ with  $q : Z \setminus \{ e \} \to H(W)$ a continuous map. Thus, the space of
continuous maps $C(Z \setminus \{ e \}, H(W))$ with the obvious  group structure is isomorphic as a
group with $H(X)$. Notice that we need not worry about behavior as $z $ approaches $e$ in $Z$ since all of
$\{ e \} \times W$ is smashed to a point. Thus, the isomorphism is topological if we choose an increasing
sequence $\{ K_n \}$ of compacta in $Z $ with union $Z \setminus \{ e \}$ and define
the metric $d(q_1,q_2) = \sup_n \ 2^{-n} d_H(q_1|_{K_n},q_2|_{K_n})$ with $d_H$ the uniform metric on $H(W)$.

In particular, the constant maps $q$ yield $H(W)$ as a subgroup of $H(X)$. Since $W$ is a manifold of
positive dimension, it follows that the path component of the identity in $H(W)$ and hence in $H(X)$ are nontrivial
subgroups. The remaining path components are cosets and so are nontrivial as well.

Let $h \in H(X)$. Given $\ep > 0$ and
$(z,w) \in X \setminus \{ e_X \}$, there is an $\ep$-chain $z_0, \dots, z_N$ for $1_Z$ with
$z_0 = z$ and $z_N = e$ and $z_i \not= e$ for $i < N$.
That is,
$d(z_{n+1},z_n) \leq \ep$. Define $\{w_0, \dots , w_N \}$ by $w_0 = w$ and $h(z_i,w_i) = (z_i,w_{i+1})$ for $i < N$.  Clearly, $(z_0,w_0), \dots , (z_{N-1}, w_{N-1}), e_X$ is an $\ep$-chain for $h$ from $(z,w)$ to $e_X$.
Similarly, there is an $\ep$-chain for $h^{-1}$ from $(z,w)$ to $e_X$. It follows that $h$ is chain recurrent.
Since $X$ is connected, $h$ is chain transitive.
\end{ex}

Now let $(Y,e)$ be a pointed space such that $Y$ and $Y \setminus \{ e \}$ are connected and, for every $y \in Y$,
the open set $Y \setminus \{ y \}$ has only finitely many components. Let $C$ be a zero-dimensional space,
and let
$(X,e_X) = (Y,e) \# C$. By assumption on $Y$, the set $X \setminus \{(y,c)\}$
has only finitely many components for $y \in Y \setminus \{ e \}, c \in C$.
Since $C$ is zero-dimensional, the components of $X \setminus \{ e_X \}$ are the sets
$\{ (Y \setminus \{ e \}) \times \{ c \} : c \in C \}$. Thus, if $C$ is infinite,
$X \setminus \{ e_X \}$ has infinitely many components, while
$X \setminus \{(y,c)\}$ has only finitely many components for $y \in Y \setminus \{ e \}, c \in C$.
Hence, if $C$ is infinite, then any
homeomorphism of $X$ fixes $e_X$.  We will assume that, even with $C$ finite, the space $Y$ is such that the point $e_X$ is
fixed by every homeomorphism of $X$.
It follows that for any homeomorphism on $h$ on $X$, the projection $\pi_2 : X \setminus \{ e_X \} \to C$
 maps the restriction of $h$ on $X \setminus \{ e_X \}$ to
a homeomorphism on $C$. Thus, we obtain $(\pi_2)_* : H(X) \to H(C)$ a continuous, surjective homomorphism of
topological groups. This splits via the continuous injection $j : H(C) \to H(X)$ given by
$j(k) = (1_Y,e) \# k$.
Now suppose that $h \in H(X)$ is in the kernel of $(\pi_2)_*$, that is, it projects to $1_C$. This means that
every $(Y \setminus \{ e \}) \times \{ c \}$ is $h$-invariant.  It follows that $h(y,c) = (q(c)(y),c)$, where
$q : C \to H(Y,e)$ is a continuous map. That is, the kernel is $C(C,H(Y,e))$ with the obvious topological group
structure.  Thus, $H(X)$ is the semi-direct product of $H(C)$ with $C(C,H(Y,e))$. The adjoint action of $j(H(C))$ is just
the action
$H(C) \times C(C,H(Y,e)) \to C(C,H(Y,e))$  given by $(k,q) \mapsto q \circ (k^{-1})$.

\begin{ex}\label{ex4.03} We construct a space such that every homeomorphism is chain transitive and the homeomorphism
group is isomorphic to the homeomorphism group of the Cantor set.

Let $(Z,e)$ be a pointed chaotic space as before, and let $(X,e_X) = (Z,e) \# C$ with
$C$ a zero-dimensional space. We first check that even if $C$ is finite, any homeomorphism $h$ on $X$ fixes $e_X$.
If not, then there exist points $x_1,x_2 \in Z$, distinct from each other and distinct from $e$, and points $a_1, a_2 \in C$
such that $h(x_1,a_1) = (x_2,a_2)$. This implies $h$ induces a homeomorphism between sufficiently
small neighborhoods $U_1$ of $x_1$ and $U_2$ of $x_2$. Choosing these as disjoint neighborhoods, we obtain a contradiction
of Condition ($\ast$).

In this case $H(Z,e) = H(Z) = \{ 1_Z \}$. That is, the group $H(Z,e)$ is trivial, and
so the group $C(C,H(Z,e))$ is trivial.  This means that $(p_2)_* : H(X) \to H(C)$ and $j : H(C) \to H(X)$ are
inverse isomorphisms, so every homeomorphism on $X$ is mapped by $\pi_1$ to $1_Z$. Just as in
Example \ref{ex4.02}, it follows that every homeomorphism is chain transitive.

When $C$ is a Cantor set, we obtain an example with homeomorphism group isomorphic to the homeomorphism group of
the Cantor set.
\end{ex}

Because of the rigidity of the connecting links, it is not true in these examples that $H(X)$ acts transitively on $X$.
We can obtain examples which satisfy this additional condition by using the beautiful construction of Slovak spaces
due to Downarowicz, Snoha, and Tywoniuk in \cite{D-S-T}.

Let $g$ be a totally transitive homeomorphism on a Cantor set $W$. That is, $g^n$ is topologically transitive for
all $n \in \Z \setminus \{ 0 \}$.
The construction begins with the \emph{suspension} of $g$.  That is, let $Y = W \times [0,1]$
with $(x,1)$ identified with $(g(x),0)$ for all $x \in W$. On $Y$ we define the real flow $\phi : \R \times Y \to Y$, the associated time $t$ map $\phi^t : Y \to Y$, and the path map $\phi_x : \R \to Y$ for $t \in \R, x \in W$ by
\begin{equation}\label{4.02}
\begin{gathered}
\phi(t,(x,s)) \ = \ \phi^t(x,s) \ = \ (g^{[ t + s ]}(x), \{ t + s \}),\\
\phi_x(t) \ = \ \phi^t(x,0),
\end{gathered}
\end{equation}
where $[ a ] $ and $\{ a \}$ are the integer part and fractional part, respectively, of the real number $a$.
Identifying $W$ with $W \times \{ 0 \} \subset Y$, we see that $g$ on $W$ is identified with the time-one map
$\phi^1$ restricted to $W$.

Observe that the flow $ \phi$ restricts to a homeomorphism from $[-\frac{1}{3},\frac{1}{3}] \times W$ onto a
neighborhood of $W \times \{ s\} $ in $Y$ for $s \in [0,1]$. It follows that the path components of $Y$
are exactly the $\R$-orbits of the flow. For $x \in W$ we let $\R x = \phi_x(\R)$ denote the $\R$-orbit through
$(x,0)$.

In $Y$ there are three types of path components:
\begin{itemize}
\item[Type 1] If $x$ is a periodic point for $g$, then $\R x$ is an circle embedded in $Y_0$. This is a
\emph{circle type} path component.

\item[Type 2] If $x$ is neither recurrent for $g$ nor for $g^{-1}$, i.e. $x \not\in \om g(x) \cup \al g(x)$, then
$\R x$ is an embedded copy of $\R$. That is, $\phi_x$ is a homeomorphism from $\R$ onto its image
in $Y$. This is an \emph{embedded $\R$ type}  path component.

\item[Type 3] If $x$ is not periodic but $x \in \om g(x) \cup \al g(x)$, then $\phi_x$ is a continuous
injection which is not a homeomorphism onto its image. In fact, $\R x \subset Y_0$ is not locally connected. This is
an \emph{injected $\R$ type}  path component.
\end{itemize}

The Slovak space construction is based on the following result:

\begin{theo}[{\cite[Lemma 4.3]{D-S-T}}]\label{theo4.04} Let $f$ be a homeomorphism on a space $Y$ with $y_0$ a point of $Y$ which is not
a periodic point for $f$. Let $\{ a_n : n \in \Z \}$ be a sequence of positive reals such that $\Sigma_n \ a_n = 1$
and the set $\{ | \ln (a_n) - \ln(a_{n-1})| : n \in \Z \}$ is bounded, and let $u : Y \setminus \{ y_0 \} \to [0,1]$ be a
continuous function. Let $Y' = Y \setminus \OO f_{\pm}(y_0)$, and on $Y'$ define the
continuous function $u' = \Sigma_n \ a_n u \circ f^n$ so that $u'$ is a closed subset of $Y' \times [0,1]$
with the first coordinate projection $\pi : u' \to Y'$ a homeomorphism. Define on the set $u'$ the homeomorphism
$f' = (\pi)^{-1} \circ f \circ (\pi)$.
Let $X$ be the closure of $u'$ in $Y \times [0,1]$.

The homeomorphism $f'$ and its inverse are uniformly continuous on $u'$, and so $f'$ extends to a homeomorphism
$h$ on $X$. The first coordinate projection $\pi : X \to Y$ maps $h$ on $X$ to $f$ on $Y$. If $y \in Y'$, then $(y,u'(y))$ is
the unique point of $X$ which is mapped by $\pi$ to $y$. Hence, $h$ is an almost one-to-one extension of $f$.
\end{theo}

Now, following \cite{D-S-T}, we apply this construction to the situation above.

Let $x_0 $ be an element of the dense
$G_{\del }$ set  $\bigcap \ \{ \Trans(g^n) : n \in \Z \setminus \{ 0 \} \}$. By \cite{A-93} Proposition 6.3 (a)
the set of $\tau \in (0, \infty)$ such that $\om \phi^{n \tau }(x_0) = Y = \al \phi^{n \tau }(x_0)$ is residual in
$(0,\infty)$. Fix such a $\tau \in (0,1)$, and let $f = \phi^{\tau }$. Thus, $x_0$ is a transitive point for
$f^n$ on $Y$ for all $n \in \Z \setminus \{ 0 \}$.

 We first define $u$ on a piece of the orbit of the point $y_0 = (x_0,0)$:
\begin{equation}\label{4.03}
u(\phi_{x_0}(t)) \ = \
\begin{cases}
0 & \text{for } - \frac{1}{2} \leq t < 0, \\
\frac{1}{2}(1 - \cos(\frac{\pi}{t})) & \text{for } 0 < t \leq \frac{1}{2}.
\end{cases}
\end{equation}
Apply the Tietze Extension Theorem to obtain the continuous function
$u : Y \setminus \{ y_0 \} \to [0,1]$.

Apply Theorem \ref{theo4.04} to $Y$ with the homeomorphism $f$. We obtain an almost one-to-one lift
$h$ on $X = \ol{u'} \subset Y \times [0,1]$. All of the path components of $X$ are mapped by
$\pi$ homeomorphically onto the path components of $Y$, except that the Type 3 path component $\R x_0$ is
cut into a sequence of path components of a new type.
\begin{itemize}
\item[Type 4] The path component $Comp_n$ of $h^n(y_0)$ is mapped by $\pi$ onto the
set $ \phi_{x_0}(((n-1)\tau, n \tau ])$. As $t \searrow (n-1) \tau$, above the open interval end,
there is a topologist's sine which projects homeomorphically. Above the $n \tau$ endpoint there is a vertical
segment $J_n = \{ h^n(y_0)\} \times [0,a_{-n}]$ to which the oscillating end of the path component $Comp_{n+1}$
converges. That is, $J_n = Comp_n \cap \ol{Comp_{n+1}}$.
Thus, each  path component $Comp_n$ is a homeomorphic image of $\R_+ = [0, \infty)$. Each is
an \emph{embedded $\R_+$}  path component.
\end{itemize}

\begin{theo}\label{theo4.05} The homeomorphism group of $X$ is $H(X) = \{ h^n : n \in \Z \}$. For all $n \not= 0$
the homeomorphism $h^n$ is topologically transitive. \end{theo}

\begin{proof}
If $n \not= 0$, then by choice of $\tau$, the homeomorphism $f^n = \phi^{n \tau}$ is topologically transitive on $Y$.
Because $\pi$ mapping $h$ to $f$ is an almost one-to-one lift, it follows that $h^n$ is topologically transitive on $X$.

If $h_1$ is any homeomorphism on $X$, then the Type 4 component $Comp_n$ is mapped to some Type 4 component
$Comp_{n+k}$. Furthermore, $J_n = Comp_n \cap \ol{Comp_{n+1}}$ is mapped to $J_{n+k} $. It follows
that $\pi$ projects $h_1$ to a continuous map on $Y$ which agrees with $f^k$ on the dense set $\OO (f_{\pm}) (y_0)$.
It follows that it projects to $f^k$. Since $\pi$ is almost one-to-one, it follows that $h_1 = h^k$.
\end{proof}

Downarowicz, Snoha, and Tywoniuk begin with $g$ on $W$ minimal and observe that, for a residual set of positive reals $\tau$,
the homeomorphism $\phi^{\tau}$ is minimal on $Y$. Choosing one such $\tau$, they have $f$ minimal on $Y$.
All of the components of $Y$ are then of Type 3. Then $h^n$ is minimal on $X$ for all $n \not= 0$.

We recall and extend
their definition of a Slovak space.

\begin{df}\label{df4.06}(a) A space $X$ is
a \emph{Slovak space} if $X$ contains at least three points, $H(X)$ is isomorphic to $\Z$ and every $h \in H(X) \setminus \{ 1_X \}$
is minimal.

(b) A space $X$ is \emph{Slovakian} if $X$ contains at least three points,
 $H(X)$ is nontrivial and every $h \in H(X) \setminus \{ 1_X \}$ is topologically transitive.
 \end{df}

 We extend Theorem 4 of \cite{D-S-T} with essentially the same proof.

 \begin{theo}\label{theo4.07} A Slovakian space is connected, and its homeomorphism group has no elements of finite order
 other than the identity. \end{theo}

\begin{proof}
Let $h$ be a topologically transitive homeomorphism on a space $X$. Suppose that $X$ contains a proper,
 clopen, nonempty subset $A$. If $h^{-1}(A) \subset A$, then for any $x \in \Trans(h)$, we have $h^n(x) \in A$ for some $n \in \N$, and so $x \in A$.
Thus, the dense set $\Trans(h)$ is contained in $A$, contradicting the assumption that $A$ is a proper clopen set. It follows
 that $B = h^{-1}(A) \setminus A$ is a proper, clopen, nonempty set, and $B \cap h(B) = \emptyset$. Define
\begin{equation}\label{4.04}
g(x) \ = \
\begin{cases}
h(x) & \text{for } x \in B,\\
h^{-1}(x) & \text{for } x \in h(B),\\
x & \text{for } x \in X \setminus(B \cup h(B)).
\end{cases}
\end{equation}
 The points of the nonempty set $B \cup h(B)$ are periodic with period $2$, and so $g \not= 1_X$. Since $g^2 = 1_X$,
 it is clear that $g$ is not topologically transitive.

Since $X$ is nontrivial and connected, it is perfect and therefore uncountable.  On such a space, no topologically
transitive homeomorphism has finite order.
\end{proof}

\begin{qs}
Does there exists a Slovakian space $X$  for which $H(X)$
is not discrete? More generally, does there exist a nontrivial space $X$ such that the topologically transitive homeomorphisms
are dense in $H(X)$? If $X$ is such a space, then every $h \in H(X)$ is chain transitive.  In particular, since
$1_X$ is chain transitive, $X$ is connected. On the other hand, $1_X$ is not topologically transitive, but it is a limit
of topologically transitive homeomorphisms, and so $H(X)$ is not discrete.
\end{qs}

The only Slovakian spaces we know of are variations on the original construction of \cite{D-S-T}. All of these
have homeomorphism group isomorphic to $\Z$.

Now we extend the above construction which was built on a totally transitive homeomorphism $g$ on a Cantor space $W$.
Suppose that $B$ is a proper, closed, $g$-invariant subset of $W$ and that $r : B \to A$ is a continuous surjection
which maps $g$ on $B$ to $1_A$, the identity on the space $A$. That is, $r^{-1}(a)$ is a closed, invariant set in $B$
for every $a \in A$. Since $B$ is proper, closed and invariant, it is disjoint from $\Trans(g)$.  In particular,
$x_0 \not\in B$.

Let $\wh B$ be the quotient of $B \times [0,1]$ in $Y$.
 Since $r \circ g = r$ on $B$, it follows that $r \circ \pi_1 : B \times [0,1] \to A$ factors to define the
 surjection $\wh r : \wh B \to A$, and each $\wh r^{-1}(a)$ is a $\phi$-invariant closed subset of $Y$.
 The preimage of $\wh B$ via the homeomorphism $\pi : u' \to Y'$ is a compact $f'$-invariant subset of $Y$. Thus,
  $\pi^{-1}(\wh B)$ is a closed $h$-invariant subset of $X$. Define
\begin{equation}\label{4.05}
\begin{split}
 E_{r,Y} &= 1_Y \cup [(\wh r)^{-1} \circ \wh r],\\
 E_{r,X} &= 1_X \cup [(\wh r \circ \pi)^{-1} \circ (\wh r \circ \pi)],
\end{split}
\end{equation}
closed equivalence relations on $Y$ and $X$, respectively. Let $q^Y_r : Y \to Y_r$ and $q^X_r : X \to X_r$ be
the projections to the quotient spaces. The homeomorphisms $f$ and $h$ induce homeomorphisms $f_r$ and $h_r$ on
the quotient spaces. The projection $\pi : X \to Y$ induces $\pi_r : X_r \to Y_r$, a continuous, almost one-to-one
surjection which maps $h_r$ to $f_r$. We will regard the homeomorphisms
induced by
$\wh r : \wh B \to A$ and $\wh r \circ \pi : \pi^{-1}(\wh B) \to A$ as identifications, so $A$ is
thought of as a subset of $Y_r$ and also as a subset of $X_r$. These are the sets
of fixed points for $f_r$ and $h_r$, respectively.

For $a \in A$ and $x \in W$, we say that $q_r^Y(\R x)$ is an \emph{$a$-orbit} if
$\om g(x) \subset r^{-1}(a)$ or $\al g(x) \subset r^{-1}(a)$. If $\om g(x) \subset r^{-1}(a)$, then
the map $q_r^Y \circ \phi_x : \R \to Y_r$ extends continuously to $\R \cup \{ + \infty \}$ by mapping
$+ \infty $ to $a$. If $x \in B$, then $x$ is an $a$-orbit iff $r(x) = a$, in which case $q_r^Y(\R x) = \{ a \}$.
If $K_1$ and $K_2$ are path components of $A$, they are \emph{linked} if there exists an $x \in W$
which is both an $a_1$-orbit and an $a_2$-orbit for some $a_1 \in K_1, a_2 \in K_2$, i.e. if $\al g(x) \subset r^{-1}(a_1)$
and $\om g(x) \subset r^{-1}(a_2)$ or vice-versa. Two path components are \emph{linkage equivalent} if there is a
finite sequence $K_1,\dots,K_N$ joining them with each $K_i$ linked to its successor.

In $Y_r$ we have a new path component type:
\begin{itemize}
\item[Type 5] Let $[K]$ be a linkage equivalence class of path components of $A$. The \emph{$[K]$-component} in $Y$ is the
union of all of the $a$-orbits for $a \in K_1 \in [K]$ and of the path components $K_1 \in [K]$.
 \end{itemize}

 It is possible for a $[K]$-component to be of embedded $\R_+$ type. The only way this can happen is if
 the embedding of $[0, \infty)$ onto the $[K]$-component maps $0$ to a point of $A$. The \emph{endpoint} of
 this $\R_+$ type component lies in $A$.

\begin{theo}\label{theo4.08} The homeomorphism group of $X_r$ is $H(X_r) = \{ h_r^n : n \in \Z \}$.
The homeomorphism $h_r^n$ is topologically transitive for all $n \not= 0$. \end{theo}

\begin{proof}
The surjection $q_r^X$ maps the totally transitive homeomorphism $h$ onto the totally transitive
homeomorphism  $h_r$.

If $h_1$ is any homeomorphism on $X_r$, we proceed just as in the proof of Theorem \ref{theo4.05}.  There is, however,
one tricky bit. It is possible that a Type 5 path component is of embedded $\R_+$ type. If $Comp_n$ maps to
$Comp_{n+k}$, then just as before, $J_n = Comp_n \cap \ol{Comp_{n+1}}$ is mapped to $J_{n+k} $, and $h_1 = h^k$ as before.

Suppose instead that $Comp_n$ is mapped to some Type 5 path component, which we call $Q_n$. Then
$Comp_{n+1}$ will have to be mapped to a Type 5 path component $Q_{n+1}$ with $J_n$ mapping to
$Q_n \cap \ol{Q_{n+1}}$, and this set contains the endpoint of $Q_n$ which is in $A$. Thus, there is a sequence
$\{ x_n \in J_n : n \in \Z \}$ with $h_1(x_n) \in A$. But the sequence $\{ x_n \}$ projects to the $f_r$-orbit of
$y_0$, and this is dense in $Y_r$. Since $\pi_r : X_r \to Y_r$ is an almost one-to-one map, the sequence
$\{ x_n \}$ is dense in $X_r$. Since the proper closed subset $A$ contains the sequence $\{ h_1(x_n) \}$ and
$h_1$ is a homeomorphism, we obtain a contradiction.
\end{proof}

If $X$ is one of the examples of Slovakian spaces as constructed above, then $X \setminus D$ is connected for any countable
subset $D$ of $X$. This is because we can choose a countable invariant subset $D_0 \subset W$ with $x_0 \in D_0$ such that
$D \subset D_0 \times [0,1]$. Since $\Trans(g) \cap \Trans(g^{-1})$ is residual and thus uncountable,
we can choose $x \in (\Trans(g) \cap \Trans(g^{-1})) \setminus D_0$.
Then the orbit $\R x$ is connected, dense in $X$, and
contained in $X \setminus D$.  We don't know whether this property holds for all Slovakian spaces.

Now for our final construction.

\begin{ex}\label{ex4.09} Let $C$ be a space which is countable or the Cantor set. There exists a space $K$
such that $H(K)$ is isomorphic as a topological group to the semi-direct product of $H(C)$ with
$C(C, \Z)$. The action of $H(K)$ on $K$ is topologically transitive, and every element of $H(K)$ is chain transitive.
If $C$ is either finite or the Cantor set, then there exists a topologically transitive homeomorphism in $H(K)$.

We begin with $g$ a totally transitive homeomorphism on a Cantor set $W$ with a fixed point
$e \in W$. Let  $r : \{e\} \to \{e\}$ be the identity. The space $Y_r$ is the suspension of $g$ with $\{ e \} \times [0,1]$ smashed
to the point $e$, and $\phi_r$ is the associated real flow on $Y_r$. Then $X_r$ is the Slovakian space over $Y_r$ with
$p : X_r \to Y_r$  the almost one-to-one projection. The group $H(X_r)$ is cyclic with generator $h_r$ mapped by $p$ to $\phi^{\tau }_r$.

Let $(K,e_K)$  be the smash product $(X_r,e) \# C$, and let $(L,e_L) = (Y_r,e) \# C$. We have the projections
\begin{equation}\label{4.06}
\begin{gathered}
\pi_L :L \setminus \{ e_L \} \ = \ (Y_r \setminus \{ e \})  \times C \ \to \ C,\\
\pi \ = \ \pi_L \circ (p \times 1_C) : K \setminus \{ e_K \} \ = \ (X_r \setminus \{ e \}) \times C \ \to \ C,
\end{gathered}
\end{equation}
and $P : (K,e_K) \to (L,e_L)$ is the smash product $(p,e) \# 1_C$.

If $C$ is a singleton, then $(K,e_K)$ is just $(X_r,e)$,
$H(C)$ is trivial, and $C(C, \Z)$ is isomorphic to $\Z$ and to $H(K)$. The identity is chain transitive, and
all other elements of $H(K)$ are topologically transitive.  So the result is clear in this case. Now assume that
$C$ has at least two points. This implies that $e_K$ disconnects $K$, while no other point does since no point
of $X_r$ disconnects $X_r$.  Hence,
every homeomorphism of $K$ preserves $e_K$.

From the discussion preceding Example \ref{ex4.03} it follows that $H(K)$ is the semi-direct product of
$H(C)$ and $C(C,H(X_r))$, which is essentially $C(C,\Z)$ since $H(X_r)$ is Slovakian. The projection
$\pi :K \setminus \{ e_K \} \to C$
induces the
group surjection $\pi_* : H(K) \to H(C)$, which is split by the injection $j :H(C) \to H(K)$.
If $ k \in H(C)$, then $j(k)$ is
$(1_{X_r},e) \# k$. In this case, the subgroup
$C(C,\Z)$ is commutative. If $q \in C(C, \Z)$, then the associated element of $H(K)$ is the projection of
$(x,c) \mapsto (h^{q(c)}_r(x),c)$. The constant elements,
the homeomorphisms  $(h^n,e) \# 1_C$, commute with all the elements of $H(K)$ since they commute with
the members of the subgroup $j(H(C))$. For
$(n,k) \in \Z \times H(C)$ let $J(n,k) = (h^n,e) \# k$. Thus,  $J : \Z \times H(C) \to H(K)$
is a topological embedding and a group homomorphism.

Notice that $P$ maps $(h^{q(c)}_r(x),c)$ to $(\phi^{q(c)\tau }_r(x),c)$ and maps $(y,k(c))$ to $(p(y),k(c))$. It follows that
every homeomorphism of $K$ projects by $P$ to a homeomorphism of $L$.

Case 1 ($C$ is finite): Let $k$ be a cyclic permutation on $C$ so that $C$ consists of a single periodic orbit under $k$.
Since $h$ is totally transitive, the product $h \times k$ on $X_r \times C$ is topologically transitive, and
so it projects to a topologically transitive element of $\Z \times H(C) \subset H(K)$.

Let $F$ be an arbitrary homeomorphism on $K$ with $k = \pi_*(F)$. Since $k$ is a permutation of a finite set, there
exists $N \in \N$ such that $k^N$ is the identity. This means that $F^N$ projects to the identity on $C$ and
so preserves each fiber $X_r \times \{ c \}$. So there exists $n_c \in \Z$ such that $F^N$ restricts to $h^{n_c}$ on
the fiber. This is chain transitive on the fiber, topologically transitive if $n_c \not= 0$, and so every point
is chain recurrent for $F^n$ and hence for $F$. Since $K$ is connected, $F$ is chain transitive on $K$.

Case 2 ($C$ is countably infinite): Since $C$ is countable, the isolated points are dense in $C$. If $x_1, x_2 \in C$ are
distinct isolated points, let $k$ interchange these two points and fix the remaining points of $C$.
By Case 1, $J(1,k) = (h,e) \# k$ restricts to a topologically transitive
homeomorphism on the closed subset $(X_r,e) \# \{ x_1, x_2 \}$.
This implies that $H(K)$ acts in a topologically transitive manner on the  set $(X_r \setminus \{e \}) \times \Iso(C)$.
This is a dense open subset of $K$, and so $H(K)$ is topologically transitive on $K$.

Case 3 ($C$ is a Cantor set): A homeomorphism $k$ on $C$ is \emph{topologically mixing} if, for all nonempty open
sets $U_1,U_2 \subset C$, there exists $N \in \N$ such that $N_k(U_1,U_2)$ contains every $n \geq N$. The shift
homeomorphism on $\{ 0,1 \}^{\Z}$ is topologically mixing, and the product of a topologically mixing and a topologically
transitive homeomorphism is topologically transitive.  It follows that if $k \in H(C)$ is topologically mixing,
then $J(1,k)$ is a topologically transitive element of $H(K)$.

It remains to show that every element of $H(K)$ is chain transitive. It suffices to show that, for
an arbitrary $F \in H(K)$ and an arbitrary point $y$ of $K$, we have $e_K \in \CC F(y)$. Applying this to
$F$ and $F^{-1}$, we see that every point of $K$ is chain recurrent for $F$.

With the homeomorphism $g$ on $W$ chosen as above, we do not know
whether this is always true. We prove it by imposing further restrictions on the homeomorphism $g$ with which we began.

Call a homeomorphism  $g$ \emph{semi-minimal} if it has a fixed point $e$ and if for every $x \not= e$ the orbit
$\OO g_{\pm}(x) = \{ g^n(x) : n \in \Z \}$ is dense in $W$. Such semi-minimal homeomorphisms exist.
Topologically mixing examples of semi-minimal homeomorphisms on the Cantor set are constructed explicitly in \cite{A-15}
Theorem 4.19, and Theorem 4.16 there shows that the semi-minimal homeomorphisms form a dense $G_{\del }$ subset
of the set of those chain transitive homeomorphisms on a Cantor set which admit a fixed point. Now assume that $g$
is a  semi-minimal homeomorphism which is topologically mixing and so is totally transitive. This implies that
if $x \not= e$, then the real orbit $\R x$ is dense in $Y_r$.

Let $F \in H(K)$ and $y \in K$. We must show that $y$ chains to $e_K$. Let  $k = \pi_*(F)$ be a homeomorphism on $C$ and
$G = P_*(F)$ a homeomorphism on $L$.

The invariant
set $\om F(y)$ contains a closed subset $M$ so that the restriction of $F$ to $M$ is minimal.  Of course, $M \subset \CC F(y)$.
So it suffices to prove that $e_K \in \CC F(M)$. This is obvious if $M = \{ e_K \}$.  Now assume $M$ is not equal to $\{ e_K \}$ and
so does not contain $e_K$ since distinct minimal sets are disjoint.

Hence, $M$ is a compact subset of $K \setminus \{e_K \}$.
Since $\pi$ maps $F$ to $k$, the subset $A = \pi(M)$ of $C$ is compact and
invariant on which $k$ restricts
to a minimal homeomorphism. Let $(\wh K , e_K)  = (X_r, e) \# A $  and $(\wh L , e_K)  = (Y_r, e) \# A$,
so that $\wh K$ is a closed $F$-invariant subset of $K$ and $\wh L$ is a closed $G$-invariant subset of $L$.
The restriction of $F$ to $M \subset \wh K$ is  minimal, and so if $\wt M = P(M) \subset \wh L$, then $G$ on $\wt M$ is
minimal. On $\wt L$ (but not on $\wt K$) the homeomorphism $(\phi^t_r,e) \# 1_A$ is defined for every $t$, and each such homeomorphism
commutes with $G$. It follows that $\wt M_t = (\phi^t_r,e) \# 1_A(\wt M)$ is a $G$-invariant subset on which $G$
is minimal. If $a \in A$, there exists $((x,s),a) \in \wt M$ for some $x \in W \setminus \{ e \}$ and $s \in [0,1)$ because
$\pi_L$ maps $\wt M$ onto $A$. Because $x \not= e$, the real orbit $\R x$ is dense in $Y_r$. It follows that
$\bigcup_t \ \wt M_t$ is dense in every fiber $\pi_L^{-1}(a)$. This implies that the union of the minimal subsets
of $G$ is dense in $\wh L$. Above each $\wt M_t$ there is a minimal subset $M_t \subset \wh K$ with $P(M_t) = \wt M_t$.
Because $P$ is an almost one-to-one map, it follows that the union $\bigcup_t \ M_t$ is dense in $\wh K$. This implies that the
recurrent points for $F$ are dense in $\wh K$, and so every point of $\wh K$ is chain recurrent. Since $\wh K$ is
connected, $F$ on $\wh K$ is chain transitive and, in particular, $e_K \in \CC F(M)$, as required.
\end{ex}

\begin{remark}
Notice that an almost one-to-one lift of a chain transitive map need not be chain transitive.
Let $t$ be translation by $1$ on $\Z$ and $t^*, t^{**}$ the extensions to the one-point compactification $\Z^* = \Z \cup \{ \infty \}$
and the two-point compactification $\Z^{**} = \Z \cup \{ + \infty, - \infty \}$. The map $p : \Z^{**} \to \Z^*$ which maps
both $+ \infty$ and $- \infty$ to $\infty$ is an almost one-to-one map, but while $t^*$ is chain transitive, $t^{**}$ is not.  In general,
if $A$ is a nowhere dense, closed, invariant set which contains $|\CC f|$ for a homeomorphism $f$ on $X$, then
smashing $A$ to the point $e$ in $X/A$, the induced homeomorphism $f_A$ on $X/A$ has $e$ as the unique minimal point,
so $f_A$ is chain transitive, the projection $q_A : X \to X/A$ is almost one-to-one since $A$ is nowhere dense, and
$f$ is not chain recurrent since $|\CC f|$ is a proper subset of $X$. For example, with $X = I^2$ and $f(x,y) = (x^2,y^2)$,
the corner points are the only chain recurrent points. Let $A$ be the boundary $I \times \{ 0,1 \} \ \cup \ \{ 0,1 \} \times I$.
\end{remark}

\section{Appendix: Chaotic Spaces}\label{const}

The rigid and strongly chaotic spaces constructed in \cite{dG-W} and \cite{C-C} are dendrites or subsets of $\R^2$.
It will be convenient for our purposes to use infinite-dimensional examples.

Let $M $ be an infinite subset of $\N \setminus \{ 1 \}$, and let $m : \N \to M$ be the unique order preserving bijection.
For $n = 0, 1, \dots$, let $M^n = \{ m(2^n(2k -1)) : k \in \N \}$ so that $\{ M^n \}$ is a partition of $M$ by a pairwise
disjoint sequence of infinite sets.  For all $i \in \N$ let $S^i$ be the sphere in $\R^{i+1}$ of radius $i^{-1}$ centered
at $(i^{-1},0,\dots,0)$ so that the origin $0$ is a point of $S^i$. Let $S(n,k) = S^{m(2^n(2k - 1))}$ for $n = 0,1, \dots$ and
$k \in \N$.

Let $Z^0$ be the two-torus $S^1 \times S^1$, and let $A^0 = \{ a^0_k: k \in \N \}$ be a dense sequence of distinct points in $Z^0$. Let $Z^1$ be $Z^0$ with a copy of $S(0,k)$ attached to $Z^0$ with $0 \in S(0,k)$ identified with the
\emph{attachment point} $a^0_k$ for $k \in \N$ so that the attached spheres are disjoint in $Z^1$. Let $r^1 : Z^1 \to Z^0$
be the retraction with each $S(0,k)$ mapped to the point $a_k$. For each $k \in \N$ let $r(0,k) : Z^1 \to S(0,k)$
be the retraction mapping $Z^1 \setminus S(0,k)$ to the point $a_k$.  Since the diameters of the spheres tend to zero,
the space $Z^1$ is compact and metrizable and the retractions are continuous. The open set $Z^1 \setminus Z^0$ is dense in
$Z^1$.

For $n \geq 1$ assume that $Z^n$ has been defined with a retraction $r^n : Z^n \to Z^{n-1}$ and with retractions
$r(n-1,k) : Z^n \to S(n-1,k)$
onto the attached spheres and so that $Z^{n} \setminus Z^{n-1}$ is dense in $Z^n$.
Let $A^n = \{ a^n_k: k \in \N \}$ be a sequence of distinct points dense in $Z^n \setminus Z^{n-1}$.
Let $Z^{n+1}$ be $Z^n$ with a copy of $S(n,k)$ attached to $Z^n$ with $0 \in S(n,k)$ identified with the
attachment point $a^n_k$ for $k \in \N$ so that the attached spheres are disjoint in $Z^{n+1}$. Let $r^{n+1} : Z^{n+1} \to Z^n$
be the retraction which maps the new spheres to their attachment points, and let $r(n,k) : Z^{n+1} \to S(n,k)$ be the
retraction mapping $Z^{n+1} \setminus S(n,k)$ to $a^n_k$. Notice that for each $x \in Z^n \setminus A^n$ the
set $(r^{n+1})^{-1}(x) = \{ x \}$.

Let $Z$ be the inverse limit of the system $\{ r^{n} : Z^{n} \to Z^{n-1}, \ n \in \N  \}$. The inclusions
$i_n : Z^{n} \to Z^{m}$ with $m > n$ commute with the retractions. We obtain a limiting inclusion and so can regard
$\{ Z^n \}$ as an increasing sequence of subsets of $Z$. The projection $r_n : Z \to Z^n$ is then a retraction with
$r^n \circ r_n = r_{n-1}$. We write $r(n,k) : Z \to S(n,k)$ for the composition of retractions $r(n,k) \circ r_{n+1}$.
Notice that if we pick $z$ from the set $Z^* = Z \setminus (\bigcup_n \ Z^n)$, then there is a
unique sequence of attachment points $\{ a^n_{k_n} \}$ converging to $z$ with
$r^n(a^n_{k_n}) = a^{n-1}_{k_{n-1}}$ for all $n$. Thus, $Z^*$,
while a dense $G_{\delta}$, is totally disconnected.

Observe that $(r_{n+1})^{-1}(S(n,k))$ and $\{ a^n_k \} \cup (r_{n})^{-1}(Z^n \setminus \{ a^n_k \})$ are connected sets
with union $Z$ and which meet only at $a^n_k$. Thus, each attachment point disconnects $Z$.

If $F \subset Z$ is a finite set containing no attachment points, then $Z \setminus F$ is connected.  For any  finite
$F \subset Z$, the set $Z \setminus F$ contains only finitely many components. In fact, the number of components is
exactly one more than the number of attachment points in $F$.

\begin{lem}\label{lem5.01} If $W$ is a closed, connected, nontrivial subset of
$Z \times I^N$ such that $W \setminus A$ is connected for any
$A \subset W$ with topological dimension at most $N$,
then either $W \subset Z^0 \times I^N$ or there exists a unique attached sphere $S(n,k)$ such that $W \subset S(n,k) \times I^N$.
\end{lem}

\begin{proof}
Notice first that the dimension of $W$ is at least $N + 2$. For if $U$ is a nonempty open set with
$\ol{U}$ a proper subset of $W$, then the topological boundary of $U$ disconnects $W$ and so has dimension at least $N + 1$.
This implies that the dimension of $W$ is at least $N + 2$.

The set $W$ is not contained in $Z^* \times I^N$ since $Z^*$ is totally disconnected and the components of
$Z^* \times I^N$ have dimension $N$. Assume $W$ is not a subset of $Z^0 \times I^N$.

If $W$ meets $(Z^{n+1} \setminus Z^{n}) \times I^N$ for some $n = 0, 1, \dots $, then $X$ meets
$(S(n,k) \setminus \{a^n_k\}) \times I^N$  for some $k$.
Since $a^n_k \times I^N$ disconnects $Z \times I^N$ and no such set disconnects $W$,
it follows that $W \subset r_n^{-1}(S(n,k)) \times I^N$.
For the attachment points $a^{n+1}_j$ in $S(n,k) \setminus \{ a^n_k \}$, the sets $a^{n+1}_j\times I^N$
also disconnect $Z$. We see that
either $W \subset S(n,k) \times I^N$ or else $W \subset r_{n+1}^{-1}(b) \times I^N$, where $b$ is one of the attachment points in
$S(n,k) \setminus \{ a^n_k \}$.

If this process does not halt with $W$ contained in the product of $I^N$ with some attached sphere, then
there is a sequence of attachment points $\{b^{n_i}_{k_i}\}$ with
$W \subset  r_{n_i+1}^{-1}(b^{n_i}_{k_i}) \times I^N$ and
$r_{n_i+1}(b^{n_{i+1}}_{k_{i+1}}) = b^{n_i}_{k_i}$
with $n_i \longrightarrow \infty$. The limit of such a sequence $\{ b^{n_i}_{k_i} \}$ is a point of $Z^*$, and this would imply that
$W$ is a subset of $Z^* \times I^N$.
\end{proof}

If $A$ and $X$ are spaces, a \emph{non-null embedding} of $A$ in $X$ is a continuous injective function $j : A \to X$
which is not homotopic to a constant map in $X$.

\begin{cor}\label{cor5.02} If $S$ is a sphere of dimension at least two and
$j : S \times I^N \to Z \times I^N$ is a non-null embedding, then for a
unique pair $(n,k)$, we have $j(S \times I^N) \subset S(n,k) \times I^N$.  Furthermore, the dimension of $S $
equals the dimension of $S(n,k)$. On the other hand,  if  $j_0 : S \to S(n,k)$ is a homeomorphism, then
$j = j_0 \times 1_{I^N} : S \times I^N \to Z \times I^N$ is a non-null embedding. \end{cor}

\begin{proof}
Corollary 1 of Theorem IV.4 in \cite{H-W} says that a connected manifold of
dimension at least $N + 2$ cannot be disconnected by a subset of dimension  $N$.
Hence, Lemma \ref{lem5.01} applies to $W = j(S \times I^N)$, and so
$j(S  \times I^N) \subset S(n,k)  \times I^N$ for some pair $(n,k)$ which is clearly unique.
The space $S \times I^N$  cannot be embedded in a manifold of smaller dimension, and so
$\dim S \leq \dim S(n,k)$. If the inequality were strict then the map $j$ would be homotopically trivial since the
homotopy groups of a sphere vanish below its dimension. Hence, the dimension of $S$ must be
$m(2^n(2k - 1)) = \dim S(n,k)$.

If $j_0 : S \to S(n,k)$ is a homeomorphism, then
$j = j_0 \times 1_{I^N} : S \times I^N \to S(n,k)  \times I^N$ is not homotopically trivial.
 Since $S(n,k)$ is a retract of $Z$, the embedding $j : S  \times I^N \to Z  \times I^N$ is not homotopically trivial.
\end{proof}

Thus, we can associate to any open subset $U \subset Z$ the set $\del(U) = \{ \dim S(n,k) : S(n,k) \subset U \} \subset M$.
Since the diameters of the attaching spheres tend to zero, it follows that
if $U$ is nonempty, then $\del(U)$ is infinite.

Corollary \ref{cor5.02} says that for a sphere $S$ of dimension at
least two there exists a non-null embedding of $S \times I^N$ into
$U \times I^N$ iff $\dim S \in \del(U)$.  That is, $\del(U)$ is a topological invariant for the sets $U \times I^N$.
Observe that if $U_1$ and $U_2$ are disjoint, nonempty, open sets in $Z$, then $\del(U_1) \cap \del(U_2) = \emptyset$ since distinct
attached spaces have distinct dimensions.

If we begin by partitioning $\N \setminus \{ 1 \}$ into a pairwise disjoint sequence $\{ M_n \}$ of infinite subsets,
then we can do this construction associating $Z_n$ with $M_n$.  That is, $\del(Z_n) = M_n$. It follows that
if $U_1 \subset Z_{n_1}, U_2 \subset Z_{n_2}$ are nonempty open subsets
with $n_1 \not= n_2$, then $\del(U_1) \cap \del(U_2) = \emptyset$.

Condition ($\ast$) of Section \label{all-ct} is thus proved for this sequence of spaces.

\bibliographystyle{amsplain}

\end{document}